\documentclass[12pt]{article}

\usepackage{geometry}
\usepackage{graphicx}
\usepackage{floatrow}

\usepackage[numbers]{natbib}
\usepackage{algorithm}
\usepackage{algpseudocode}
\usepackage{url}
\usepackage{hyperref}
\usepackage[dvipsnames]{xcolor}
\usepackage{amsmath, amssymb, amsthm}
\usepackage{verbatim}

\theoremstyle{plain}
\newtheorem{thm}{Theorem}[section]

\newtheorem{lem}[thm]{Lemma}
\newtheorem{prop}[thm]{Proposition}
\newtheorem{corol}{Corollary}[section]
\theoremstyle{definition}

\theoremstyle{remark}

\newcommand{\pr}{\mathsf{P}}
\newcommand{\dd}{\mathrm{d}}
\newcommand{\vecop}{\operatorname{vec}}

\newcommand{\veps}{\varepsilon}
\newcommand{\tr}{\operatorname{tr}}
\newcommand{\etr}{\operatorname{etr}}
\newcommand{\diag}{\operatorname{diag}}

\newcommand{\tsp}{\mathsf{T}}

\newcommand{\rN}{\mathcal{N}}
\newcommand{\rW}{\mathcal{W}}

\newcommand{\rM}{\mathcal{M}}

\newcommand{\emax}{\gamma_{\max}}
\newcommand{\emin}{\gamma_{\min}}

\newcommand{\Alpha}{\mathcal{A}}
\newcommand{\Beta}{\mathcal{B}}

\newcommand{\R}[1]{\mathbb{R}^{#1}}
\newcommand{\SPD}[1]{\mathbb{S}^{#1}_{++}}
\newcommand{\SPSD}[1]{\mathbb{S}^{#1}_{+}}

\begin{document}
\title{\vspace{-1cm}Convergence Analysis of a Collapsed Gibbs Sampler for Bayesian Vector
 Autoregressions}

\author{Karl Oskar Ekvall \\
  \normalsize{Division of Biostatistics, Institute of Environmental Medicine,
  Karolinska Institute\footnote{Substantial parts of the work was done while the
  author was at the University of Minnesota and Vienna University of
  Technology.}} \\
   \normalsize{{\tt karl.oskar.ekvall@ki.se}}
   \and Galin L. Jones \\
   \normalsize{School of Statistics, University of Minnesota} \\
   \normalsize{{\tt galin@umn.edu}}
}
\date{}
\maketitle

\begin{abstract}
  We study the convergence properties of a collapsed Gibbs sampler for Bayesian
  vector autoregressions with predictors, or exogenous variables. The Markov
  chain generated by our algorithm is shown to be geometrically ergodic
  regardless of whether the number of observations in the underlying vector
  autoregression is small or large in comparison to the order and dimension of
  it. In a convergence complexity analysis, we also give conditions for when
  the geometric ergodicity is asymptotically stable as the number of
  observations tends to infinity. Specifically, the geometric convergence rate
  is shown to be bounded away from unity asymptotically, either almost surely
  or with probability tending to one, depending on what is assumed about the
  data generating process. This result is one of the first of its kind for
  practically relevant Markov chain Monte Carlo algorithms. Our
  convergence results hold under close to arbitrary model misspecification.
\end{abstract}

\newpage

\section{Introduction}
\label{sec:intro}
Markov chain Monte Carlo (MCMC) is often used to explore the posterior
distribution of a vector of parameters $\theta$ given data $\mathcal{D}$. To
ensure the reliability of an analysis using MCMC it is essential to understand the
convergence properties of the chain in use \citep{Doss.etal2014,
Flegal.etal2008, Flegal.Jones2011, Geyer1992, Jones.Hobert2001,
Latuszynski.etal2013, Vats.etal2019} and, accordingly, there are numerous
articles establishing such properties for different MCMC algorithms
\citep[e.g.][]{Abrahamsen.Hobert2017, Backlund.Hobert2020, Hobert.etal2018,
Jarner.Hansen2000, Johnson.Jones2015, Jones.Hobert2004, Qin.Hobert2019,
Roberts.Tweedie1996, Vats2017}. It has been common in this literature to treat
the data $\mathcal D$ as fixed, or realized. Thus, the model for how the data
are generated has typically been important only insofar as it determines the
likelihood function based on an arbitrary realization---the stochastic
properties of the data prescribed by that model have not been emphasized. This
is natural since the target distribution, i.e. the posterior distribution,
treats the data as fixed. On the other hand, due to the rapid growth of data
available in applications, it is also desirable to understand how performance is
affected as the number of observations increases. When this happens, the data
are more naturally thought of as stochastic; each time the sample size increases
by one, the additional observation is randomly generated. The study of how
convergence properties of MCMC algorithms are affected by changes in the data is
known as convergence complexity analysis \citep{Rajaratnam.Sparks2015} and it
has attracted increasing attention recently \citep{Johndrow.etal2018,
Qin.Hobert2019, Qin.Hobert2019b, Yang.Rosenthal2019, Yang.etal2016}.

Accounting for randomness in the data and letting the sample size grow leads to
a more complicated analysis than when the data are fixed. In fact, to date,
convergence complexity analysis has only been successfully carried out for a
few, maybe even one \cite{Qin.Hobert2019}, practically relevant MCMC algorithms,
by which we mean MCMC algoritms for settings where one could not easily sample
by less complicated methods. We study a MCMC algorithm for a fundamental model
in time series analysis: a Bayesian vector autoregression with predictors
(VARX). Briefly, the VARX we consider assumes that $Y_t \in \R{r}$ and $X_t\in
\R{p}$ satisfy, for $t = 1, \dots, n$,
\begin{equation}\label{eq:varx}
  Y_t = \sum_{i = 1}^q\Alpha_i^\tsp Y_{t - i} + \Beta^\tsp X_t + U_t
\end{equation}
with $U_1, \dots, U_n$ independent and multivariate normally distributed with
mean zero and common covariance matrix $\Sigma \in \SPD{r}$, the set of $r\times
r$ symmetric and positive definite (SPD) matrices, $\Alpha_i \in \R{r\times r}$,
$i = 1, \dots, q$, and $\Beta \in \R{p\times r}$. As we will see, the likelihood
for the VARX is closely related to that of a multivariate linear regression and,
consequently, our results apply also to that model. Moreover, \eqref{eq:varx}
simplifies to a multivariate linear regression if $\Alpha_i = 0$ for all $i$. We
focus on the VARX because it motivates our study and, in particular, our choice
of priors. More details on the model specification, priors, and resulting
posterior distributions are given in Section~\ref{sec:BVARX}.

The target distribution of our algorithm is the posterior distribution of
$\theta = (\Alpha_1, \dots, \Alpha_q, \Beta, \Sigma)$ given $\mathcal D =
\{(Y_1, X_1), \dots, (Y_n, X_n)\}$. We will consider both fixed and growing data
and refer to the two settings as the small-$n$ and large-$n$ setting,
respectively. By $n$ being small we mean that it is fixed and possibly small in
comparison to $r$ and $q$, but $n > p$ throughout. Many large VARs in the
literature \citep{Banbura.etal2009, Ghosh.etal2018, Koop2013} are covered by
this setting. By $n$ being large we mean that it is increasing and that the data
are stochastic.

The algorithm we consider is a collapsed Gibbs sampler that simplifies to a
commonly considered (non-collapsed) Gibbs sampler when there are no predictors
in the model; that is, when $\Beta \equiv 0$ in \eqref{eq:varx}. To discuss our
results, we require some more notation. Let $F(\cdot |\mathcal{D})$ denote the
VARX posterior distribution having density $f(\theta | \mathcal{D})$ with
support on $\Theta \subseteq \R{d}$ for some $d \geq 1$ and let $K^h$ ($K \equiv
K^1$) be the $h$-step transition kernel for a Markov chain with state space
$\Theta$, started at a point $\theta \in \Theta$. We assume throughout that all
discussed Markov chains are irreducible, aperiodic, and Harris recurrent
\citep{Meyn.Tweedie2011}, and that sets on which measures are defined are
equipped with their Borel $\sigma$-algebra. Our analysis is focused on
convergence rates in total variation distance, by which we mean the rate at
which $\Vert K^h(\theta, \cdot) - F(\cdot | \mathcal{D})\Vert_{TV}$ approaches
zero as $h$ tends to infinity, where $\Vert \cdot\Vert_{TV}$ denotes the total
variation norm. If this convergence happens at a geometric (or exponential)
rate, meaning there exist a $\rho \in [0, 1)$ and an $M:\Theta \to [0, \infty)$
such that for every $\theta \in \Theta$ and $h \in \{1, 2, \dots\}$
\begin{equation} \label{eq:geo_erg}
 \Vert K^h(\theta, \cdot) - F(\cdot | \mathcal{D}) \Vert_{TV} \leq
 M(\theta)\rho^h,
\end{equation}
then the Markov chain, or the kernel $K$, is said to be geometrically ergodic.
The geometric convergence rate $\rho^\star$ is the infimum of the set of $\rho
\in [0, 1]$ such that \eqref{eq:geo_erg} holds \citep{Qin.Hobert2019}. Since all
probability measures have unit total variation norm, $\rho^\star$ is always in
$[0, 1]$, and $K$ is geometrically ergodic if and only if $\rho^\star < 1$. A
substantial part of the literature on convergence of MCMC algorithms is centered
around geometric ergodicity, for good reasons: under moment conditions, a
central limit theorem (CLT) holds for functionals of geometrically ergodic
Markov chains \citep{Chan.Geyer1994, Jones2004} and the variance in the
asymptotic distribution given by that CLT can be consistently estimated
\citep{Flegal.Jones2010,Jones.etal2006, Vats.etal2018}, allowing principled
methods for ensuring reliability of the results \citep{Robertson.etal2019,
Vats.etal2019}. Our main result in the small-$n$ setting gives conditions that
ensure $\rho^\star < 1$ when the data are fixed and $K$ is the kernel
corresponding to the algorithm we consider.

Notice that, although it is suppressed in the notation, $K$, $M$, $\rho$, and,
hence, $\rho^\star$ typically depend on $\mathcal{D}$. In the large-$n$ setting,
we are no longer considering a single dataset, but a sequence of datasets
$\{\mathcal{D}_n\} := \{\mathcal{D}_1, \mathcal{D}_2, \dots\}$, where
$\mathcal{D}_n$ here denotes a dataset with $n$ observations. Consequently, for
every $n$ there is a posterior distribution $F(\cdot | \mathcal{D}_n)$ and a
Markov chain with kernel $K_n$ that is used to explore it. To each kernel $K_n$
there also corresponds a geometric convergence rate $\rho^\star_n$. Since
$\rho^\star_n$ depends on $\mathcal{D}_n$, the sequence $\{\rho^\star_n\}$ is
now one of random variables, ignoring possible issues with measurability. If the
convergence rate $\rho_n^\star$ tends to one in probability or almost surely as
$n\to \infty$, we say that the convergence rate is unstable; in practice, we
expect an algorithm that generates a chain with geometric convergence rate
tending to one to be less reliable when applied to large datasets. Thus, we are
interested in bounding $\{\rho^\star_n\}$ away from unity asymptotically, in
either one of two senses: first, if there exists a sequence of random variables
$\{\bar{\rho}_n\}$ such that $\rho^\star_n \leq \bar{\rho}_n$ for every $n$ and
$\limsup_{n\to \infty}\bar{\rho}_n < 1$ almost surely, then we say that
$\{K_n\}$ is asymptotically geometrically ergodic almost surely, or the
geometric ergodicity is asymptotically stable almost surely. Secondly, if
instead of the upper limit being less than unity almost surely it holds that
$\lim_{n \to \infty}\pr(\bar{\rho}_n < 1) = 1$, then we say that $\{K_n\}$ is
asymptotically geometrically ergodic in probability, or that the geometric
ergodicity is asymptotically stable in probability. In the large-$n$ setting, we
give conditions for asymptotically stable geometric ergodicity, in both of the
two senses, of the Markov chain generated by our algorithm.  An intuitive,
albeit somewhat loose, interpretation of our main results is that the geometric
ergodicity is asymptotically stable if the sample covariance matrix
$n^{-1}\sum_{t = 1}^n[Y_t^\tsp, \dots, Y_{t - q}^\tsp, X_t^\tsp][Y_t^\tsp,
\dots, Y_{t - q}^\tsp, X_t^\tsp]^\tsp$ tends to some positive definite limit. In
particular, as long as this holds, the geometric ergodicity of the Markov chain
we study is asymptotically stable under arbitrary model misspecification. Our
small- and large-$n$ results are complementary: the small-$n$ results establish geometric
ergodicity for fixed $n$ but do not ensure asymptotic geometric ergodicity,
while the laarge-$n$ results give asymptotic geometric ergodicity but do not imply
geometric ergodicity for fixed $n$.

The rest of the paper is organized as follows. We begin in
Section~\ref{sec:BVARX} by completing the specification of the model and priors.
Because some of the priors may be improper we derive conditions which guarantee
the posterior exists and is proper. In Section \ref{sec:collapsed} we develop a
collapsed Gibbs sampler for exploring the posterior. Conditions for geometric
ergodicity for small $n$ are presented in Section \ref{sec:geo_erg} and
conditions for asymptotically stable geometric ergodicity are given in Section
\ref{sec:asy_stab}. Some concluding remarks are given in Section
\ref{sec:var_disc}.

\section{Bayesian vector autoregression with predictors} \label{sec:BVARX}
Recall the definition of the VARX in \eqref{eq:varx}. To complete the
specification, we assume that the starting point $(Y_{- q + 1}, \dots, Y_0)$ is
non-stochastic and known and that the predictors are strongly exogenous. By the
latter we mean that $\{X_t\}$ is independent of $\{U_t\}$ and has a distribution
that does not depend on the model parameters. With these assumptions the
following lemma is straightforward. Its proof is provided in Appendix
\ref{app:main} for completeness. Let $Y = [Y_1, \dots, Y_n]^\tsp \in \R{n \times
r}$, $X = [X_1, \dots, X_n]^\tsp \in \R{n \times p}$, $Z_t = [Y_{t - 1}^{\tsp},
\dots, Y_{t - q}^\tsp]^\tsp \in \R{qr}$, $t = 1, \dots, n$, and $Z = [Z_1,
\dots, Z_n]^\tsp \in \R{n \times qr}$. Let also $\Alpha = [\Alpha_1^\tsp, \dots,
\Alpha_q^\tsp]^\tsp \in \R{qr \times r}$ and $\alpha = \vecop(\Alpha)$, where
$\vecop(\cdot)$ is the vectorization operator, stacking the columns of its
matrix argument.

\begin{lem} \label{lem:joint_dens}
 The joint density for $n$ observations in the VARX is $f(Y, X\mid \Alpha, \Beta,
\Sigma) = f(X)f(Y\mid X, \Alpha, \Beta, \Sigma)$ with $f(Y \mid X, \Alpha,
\Beta, \Sigma)$ given by
 \begin{align*}
  \begin{split}
   (2\pi)^{-nr/2} \vert \Sigma\vert^{-n/2}\etr\left[-\frac{1}{2}\Sigma^{-1} (Y -
   Z\Alpha - X\Beta)^\tsp(Y - Z\Alpha - X\Beta)\right]
  \end{split},
 \end{align*}
 where $\etr(\cdot) = \exp(\tr(\cdot))$.
\end{lem}

For fixed data, $f(Y\mid X, \Alpha, \Beta, \Sigma)$ is the same as for a
multivariate linear regression with partitioned design matrix $[Z, X] \in
\R{n\times (qr + p)}$ and coefficient matrix $[\Alpha^\tsp, \Beta^\tsp]^\tsp$.
However, our choice of priors is guided by the vector autoregression. Let
$\SPSD{r}$ denote the set of $r \times r$ symmetric positive semi-definite
(SPSD) matrices and, to define priors, let $m \in \R{qr^2}$, $C \in
\SPSD{qr^2}$, $D \in \SPSD{r}$, and $a \geq 0$ be hyperparameters. Our prior on
$\theta = (\alpha, \Beta, \Sigma) \in \Theta = \R{qr^2} \times \R{p\times r}
\times \SPD{r}$ is of the form $f(\theta) = f(\alpha)f(\Beta)f(\Sigma)$, with
\begin{align*}
 f(\alpha) & \propto \exp\left(-\frac{1}{2}[\alpha - m]^\tsp C [\alpha - m]
 \right), \\ f(\Beta)  & \propto 1,
\end{align*}
and
\begin{align*}
 f(\Sigma) \propto \vert \Sigma\vert^{-a/2}\etr\left(-\frac{1}{2}D \Sigma^{-1}
 \right),
\end{align*}
where $\vert \cdot \vert$ means the determinant when applied to matrices. The
flat prior on $\Beta$ is standard in multivariate scale and location problems,
including in particular the multivariate regression model. The priors on
$\alpha$ and $\Sigma$ are common in macroeconomics \citep{Karlsson2013} and the
prior on $\Sigma$ includes the inverse Wishart ($D\in \SPD{r}, a > 2r$) and
Jeffreys prior ($D = 0, a = r + 1$) as special cases. In other work on similar
models it is often assumed that $[Z, X]$ has full rank or that the prior for
$[\Alpha^\tsp, \Beta^\tsp]^\tsp$ is proper \citep{Abrahamsen.Hobert2017,
Backlund.Hobert2020,Ghosh.etal2018,Hobert.etal2018, Korobilis2008,
Tiao.Zellner1964}. Treating $\Alpha$ and $\Beta$ differently is appealing in the
current setting: it adheres to the common practice of using flat priors for
regression coefficients such as $\Beta$, while $m$ and $C$ can be chosen to
reflect the fact that many commonly studied time series are known to be near
non-stationary in the unit root sense. In particular, with economic and
financial data it is often reasonable to expect the diagonal elements of
$\Alpha_1$ to be near one. For a more thorough discussion of popular priors in
Bayesian VARs we refer the reader to Karlsson \citep{Karlsson2013}.

The following result gives two different sets of conditions that lead to a
proper posterior. Though we focus on proper normal or flat priors for $\alpha$
in the rest of the paper, it may be relevant for other work to note that the
proposition holds for any prior $f(\alpha)$ satisfying the conditions. For
example, $f(\alpha)$ could be truncated to impose stability of the VARX, which
if $q = 1$ corresponds to a prior with support only on $\alpha$ for which the
spectral norm of $\Alpha$ is less than one \citep[see e.g.][for
definitions]{Lutkepohl2005}.

\begin{prop} \label{prop:prop_post}
 If either
 \begin{enumerate}
  \item $D \in \SPD{r}$, $X$ has full column rank, $n + a > 2r + p$, and
  $f(\alpha)$ is proper; or
  \item $[Y, Z, X] \in \R{n \times (r + qr + p)}$ has full column rank, $n + a >
  (2 + q)r + p$, and $f(\alpha)$ is bounded,
 \end{enumerate}
 then the posterior distribution is proper; if, in addition, $f(\alpha) \propto
\exp(-[\alpha - m]^\tsp C [\alpha - m]/2)$, then with $S = n^{-1}(Y - Z\Alpha -
X\Beta)^\tsp (Y - Z\Alpha - X\Beta)$, the posterior density is characterized by
 \begin{align} \label{eq:post}
  f(\Alpha, \Beta, \Sigma \mid Y, X) \propto \vert\Sigma\vert^{-\frac{n +
  a}{2}}\etr\left(-\frac{1}{2}\Sigma^{-1}[D + nS] - \frac{1}{2}(\alpha - m)^\tsp
  C(\alpha - m)\right).
 \end{align}
\end{prop}

\begin{proof}
 Appendix \ref{app:main}.
\end{proof}

The first set of conditions is relevant to the small $n$-setting. It implies
that if the prior on $\Sigma$ is a proper inverse Wishart density, so that $a >
2r$ and $D\in \SPD{r}$, then the posterior is proper if $f(\alpha)$ is proper
and $X$ has full column rank. In particular, $r$ or $q$ can be arbitrarily large
in comparison to $n$. Thus, this setting is compatible with large VARs
\citep{Banbura.etal2009,Ghosh.etal2018,Koop2013}. The second set of conditions
allows for the use of improper priors also on $\alpha$ and $\Sigma$ when $n$ is
large in comparison to all of $p$, $q$, and $r$. The full column rank of $[Y, Z,
X]$ is natural in large-$n$ settings. In practice, one expects it to hold unless
the squares regression of $Y$ on $Z$ and $X$ gives residuals that are
identically zero.

The literature on convergence properties of MCMC algorithms for Bayesian VAR(X)s
is limited. An MCMC algorithm for a multivariate linear regression model has
been proposed and its convergence rate in the small-$n$ setting studied
\citep{Hobert.etal2018}. By the preceding discussion, this includes the VARX as
a special case, however, the (improper) prior used is $f(\theta) \propto \vert
\Sigma\vert^{-a}$ which is not compatible with the large VARXs we allow for in
the small-$n$ setting. A two-component ($\Alpha$ and $\Sigma$) Gibbs sampler for
Bayesian vector autoregressions without predictors has been proposed
\citep{Kadiyala.Karlsson1997}, but the analysis of it is simulation-based and as
such does not provide any theoretical guarantees. Our results address this
since, as we will discuss in more detail below, the algorithm we consider
simplifies to this Gibbs sampler when there are no predictors.

\section{A collapsed Gibbs sampler} \label{sec:collapsed}

If the precision matrix in the prior on $f(\alpha)$, $C$, is a matrix of zeros,
then the VARX posterior is a normal-(inverse) Wishart for which MCMC is
unnecessary. However, when $C \in \SPD{qr^2}$ the posterior is analytically
intractable and there are many potentially useful MCMC algorithms. For example,
the full conditional distributions have familiar forms so it is straightforward
to implement a three-component Gibbs sampler. Another sensible option is to
group $\Alpha$ and $\Beta$ and update them together. Here, we will instead make
use of the particular structure the partitioned matrix $[Z, X]$ offers and
devise a collapsed Gibbs sampler \citep{Liu1994}; that is, a Gibbs sampler where
some updates are not using full conditional distributions, but conditional
distributions with one or more components integrated out. As we will see, the
structure of the collapsed sampler, and in particular relations between the
convergence rates of marginal chains and the full chain, is instrumental to our
development. For a discussion more generally of how operator norms of collapsed
and non-Collapsed Gibbs samplers compare we refer the reader to
\citep{Liu1994}.

For the case where the precision matrix $C$ in the prior on $\alpha$ is positive
definite and $\Beta \equiv 0$, so that the predictor matrix $X$ plays no role in
the model, a two-component Gibbs sampler has been proposed \citep{Karlsson2013}.
We will show that the algorithm we study includes this Gibbs sampler as a
special case, with minor modifications, and as a consequence our results apply
almost verbatim to that sampler. A formal description of one iteration of the
collapsed Gibbs sampler is given in Algorithm \ref{alg:collapsed}.

\begin{algorithm}[H]
 \caption{Collapsed Gibbs sampler} \label{alg:collapsed}
 \begin{algorithmic}[1]
  \State {\it Input:} Current value $(\alpha^h, \Beta^h, \Sigma^h)$
  \State Draw $\Sigma^{h + 1}$ from the distribution of
  $\Sigma \mid \alpha^{h}, Y, X$
  \State Draw $\alpha^{h + 1}$ from the distribution of
  $\alpha \mid \Sigma^{h + 1}, Y, X$
  \State Draw $\Beta^{h + 1}$ from the distribution of
  $\Beta \mid \alpha^{h + 1}, \Sigma^{h + 1}, Y, X$
  \State Set $h=h+1$
 \end{algorithmic}
\end{algorithm}
We next give the conditional distributions necessary for its implementation. Let
$\rM(M, U, V)$ denote the matrix normal distribution with mean $M$ and scale
matrices $U$ and $V$, that is, the distribution of a matrix whose vectorization
is multivariate normal with mean $\vecop(M)$ and covariance matrix $V\otimes U$,
where $\otimes$ is the Kronecker product. Let also $\rW^{-1}(U, c)$ denote the
inverse Wishart distribution with scale matrix $U$ and $c$ degrees of freedom.
For any real matrix $M$, define $P_M$ to be the projection onto its column space
and $Q_M$ the projection onto the orthogonal complement of its column space.
Finally, define $B = B(\Sigma) = C + \Sigma^{-1} \otimes Z^\tsp Q_X Z$ and $u =
u(\Sigma) = B^{-1}[Cm + (\Sigma^{-1} \otimes Z^\tsp Q_X)\vecop(Q_X Y)]$.

\begin{lem} \label{lem:collapse_dist}
 If one of the two sets of conditions in
 Proposition~\ref{prop:prop_post} holds, then
 \begin{align*}
   & \Sigma \mid \Alpha, Y, X \sim  \rW^{-1}\left(D + (Y - Z\Alpha)^\tsp Q_X (Y
     - Z\Alpha),  n + a - p - r - 1\right),  \\
   & \alpha \mid \Sigma, Y, X \sim \rN(u, B^{-1}) \text{, and } \\
   & \Beta \mid \Alpha, \Sigma, Y, X \sim \rM\left([X^\tsp X]^{-1}X^\tsp (Y -
     Z\Alpha), [X^\tsp X]^{-1}, \Sigma\right).
 \end{align*}
\end{lem}

\begin{proof}
 Appendix \ref{app:main}.
\end{proof}

The collapsed Gibbs sampler in Algorithm \ref{alg:collapsed} simulates a
realization from a Markov chain having one-step transition kernel $K_C(\theta',
A)$ defined, for any measurable $A \subseteq \Theta = \R{qr^2} \times \R{p\times
r} \times \SPD{r}$, by
\[
 \iiint I_A(\alpha, \Beta, \Sigma)f(\Sigma \mid \alpha', Y, X)f(\alpha \mid
 \Sigma, Y, X)f(\Beta \mid \alpha, \Sigma, Y, X)\,\dd \Sigma\, \dd \alpha\, \dd
 \Beta,
\]
where the subscript $C$ is short for collapsed. However, instead of working
directly with $K_C$ we will use its structure to reduce the problem in a
convenient way. Consider the sequence $\{(\alpha^h, \Sigma^h)\}$, $h = 1, 2,
\dots$, obtained by ignoring the component for $\Beta$ in Algorithm
\ref{alg:collapsed}. The sequence $\{(\alpha^h, \Sigma^h)\}$ is essentially
generated by a two-component Gibbs sampler. More precisely, if $Q_X$ is replaced
by the $n\times n$ identity $I_n$ in the conditional distributions of $\Sigma$
and $\alpha$ used in Algorithm \ref{alg:collapsed}, then the algorithm defined
by steps 1, 2, 3, and 5 is a two-component Gibbs sampler exploring the posterior
$f(\alpha, \Sigma \mid Y)$ for the model that takes $\Beta \equiv 0$ in
\eqref{eq:varx}. The transition kernel for $\{(\alpha^h, \Sigma^h)\}$ is, for
any measurable $A \subseteq \R{qr^2} \times \SPD{r}$,
\[
 K_G((\alpha', \Sigma'), A) = \iint I_A(\alpha, \Sigma) f(\alpha \mid
 \Sigma, Y, X)f(\Sigma \mid \alpha', Y, X)\,\dd \alpha\, \dd \Sigma .
\]
A routine calculation shows that since $K_G$, by construction, has invariant
distribution $F_{\Alpha, \Sigma}(\cdot \mid Y, X)$, then $K_C$ has the VARX
posterior $F(\cdot | Y, X)$ as its invariant distribution.

The sequences $\{\alpha^h\}$ and $\{\Sigma^h\}$ are also Markov chains.  The
transition kernel for the $\{\alpha^h\}$ sequence is, for any measurable
$A\subseteq \R{qr^2}$, \begin{align}\label{eq:transition_M} K_\Alpha(\alpha', A)
= \iint I_A(\alpha) f(\alpha \mid \Sigma, Y, X)f(\Sigma \mid \alpha', Y, X)\,
\dd \Sigma\,\dd \alpha. \end{align} The transition kernel, $K_{\Sigma}$, for the
$\{\Sigma^h\}$ sequence is constructed similarly.  The kernel $K_\Alpha$
satisfies detailed balance with respect to the posterior marginal
$F_{\Alpha}(\cdot \mid Y, X)$ and similarly for $K_{\Sigma}$ and hence each has
the respective posterior marginal as its invariant distribution.  However, the
kernels $K_{C}$ and $K_{G}$ do not satisfy detailed balance with respect to
their invariant distributions.

In Sections~\ref{sec:geo_erg} and~\ref{sec:asy_stab} we will establish geometric
ergodicity of $K_C$ and study its asymptotic stability, respectively. Our
approach, which is motivated by the following lemma, will be to analyze
$K_{\Alpha}$ in place of $K_C$; the lemma says we can analyze either of $K_G$,
$K_{\Alpha}$ or $K_{\Sigma}$ in place of $K_C$ (see also \citep{Robert1995}).
The proof of the lemma uses only well known results about de-initializing Markov
chains \cite{Roberts.Rosenthal2001} and can be found in Appendix \ref{app:main}.

\begin{lem} \label{lem:tv_dist}
 For any $\theta = (\alpha, \Beta, \Sigma) \in \Theta$, and $h \in \{1, 2,
 \dots\}$,
 \begin{align*}
  \Vert K_C^h(\theta, \cdot)  - F(\cdot | Y, X)\Vert_{TV} & = \Vert
  K_G^h((\alpha, \Sigma), \cdot) - F_{\Alpha, \Sigma}(\cdot|Y,X)\Vert_{TV} \\
  & \leq \Vert K_{\Alpha}^{h - 1}(\alpha, \cdot) -
  F_{\Alpha}(\cdot|Y,X)\Vert_{TV}
 \end{align*}
\end{lem}

The primary tool we will use for investigating both geometric ergodicity and
asymptotic stability is the following well known result \citep[][Theorem
12]{Rosenthal1995}, which has been specialized to the current setting.  Note
that the kernel $K_{\Alpha}$ acts to the left on measures, that is, for a
measure $\nu$, we define
\[
 \nu K^{h}_{\Alpha}(\cdot) = \int \nu(\dd \alpha) K^{h}_{\Alpha}(\alpha, \cdot).
\]

\begin{thm} \label{thm:Rosenthal}
  Suppose $V: \R{qr^2} \to [0, \infty)$ is such that for some $\lambda < 1$ and
  some $L < \infty$
  \begin{equation} \label{eq:drift_cond} \int V(\alpha)
    K_{\Alpha}(\alpha', \dd \alpha) \leq \lambda V(\alpha') + L \hspace*{5mm}
    \text{ for all } \alpha'.
 \end{equation}
 Also suppose there exists $\veps > 0$, a measure $R$, and some $T > 2L / (1 -
 \lambda)$ such that
 \begin{equation} \label{eq:minor_cond}
   K_{\Alpha}(\alpha, \cdot) \geq \veps R(\cdot) \hspace*{5mm} \text{ for all }
   \alpha \in \{\alpha : V(\alpha) \leq T\}.
 \end{equation}
 Then $K_{\Alpha}$ is geometrically ergodic and, moreover, if
 \[
  \bar{\rho} = (1 - \veps)^{c} \vee \left(\frac{1 + 2 L + \lambda T}{1 + T}
  \right)^{1 - {c}}\left(1 + 2L + 2\lambda T\right)^{c} \hspace*{5mm} \text{ for
  } c \in (0, 1),
 \]
 then, for any initial distribution $\nu$,
 \begin{equation} \label{eq:tv_ub}
  \|\nu K^{h}_{\Alpha}(\cdot) - F_{\Alpha}(\cdot | Y, X) \|_{TV} \le \left(2 +
  \frac{L}{1-\lambda} + \int V(\alpha) \nu(\dd \alpha)\right)  \bar{\rho}^{h} .
 \end{equation}
\end{thm}

It is common for the initial value to be chosen deterministically, in which case
\eqref{eq:tv_ub} suggests choosing a starting value to minimize $V$.
Theorem~\ref{thm:Rosenthal} has been successfully employed to determine
sufficient burn-in in the sense that the upper bound on the right-hand side of
\eqref{eq:tv_ub} is below some desired value \citep{Jones.Hobert2001,
Jones.Hobert2004, Rosenthal1996}, but, unfortunately, the upper bound is often
so conservative as to be of little utility. However, our interest is twofold; it
is easy to see that there is a $c \in (0,1)$ such that $\bar{\rho}<1$ and hence
if $K_{\Alpha}$ satisfies the conditions, then it is geometrically ergodic and,
as developed and exploited in other recent research \citep{Qin.Hobert2019}, the
geometric convergence rate $\rho^\star$ is upper bounded by $\bar{\rho}$.
Outside of toy examples, we know of no general state space Monte Carlo Markov
chains for which $\rho^{\star}$ is known.

Consider the setting where the number of observations tends to infinity; that
is, there is a sequence of data sets $\{\mathcal{D}_{n} \}$ and corresponding
transition kernels $\{ K_{\Alpha, n} \}$ with $n \to \infty$.  If $\liminf_{n\to
\infty}\bar{\rho}_n = 1$ almost surely, then we say the drift
\eqref{eq:drift_cond} and minorization \eqref{eq:minor_cond} are asymptotically
unstable in the sense that, asymptotically, they provide no control over
$\rho^{\star}_{n}$. On the other hand, because $\rho^{\star}_{n} \le
\bar{\rho}_{n}$ establishing that $\limsup_{n\to \infty}\bar{\rho}_n < 1$ almost
surely or that $\lim_{n \to \infty}\pr(\bar{\rho}_n < 1) = 1$, leads to
asymptotically stable geometric ergodicity as defined in the introduction.

Notice that $\bar{\rho}$ depends on the drift function $V$ through
$\varepsilon$, $\lambda$, and $L$.  Thus the choice of drift function which
establishes geometric ergodicity for a fixed $n$ may not result in asymptotic
stability as $n \to \infty$.  Indeed, in Section~\ref{sec:geo_erg} we use one
$V$ to show that $K_{\Alpha}$ is geometrically ergodic under weak conditions
when $n$ is fixed, while in Section~\ref{sec:asy_stab} a different drift
function and slightly stronger conditions are needed to achieve asymptotically
stable geometric ergodicity of $K_{\Alpha}$.

\section{Geometric ergodicity} \label{sec:geo_erg}

In this section we consider the small-$n$ setting. That is, $n$ is fixed and the
data $Y$ and $X$ observed, or realized, and hence treated as constant.
Accordingly, we do not use a subscript for the sample size on the transition
kernels.  We next present some preliminary results that will lead to geometric
ergodicity of $K_{\Alpha}$, and hence $K_G$ and $K_C$.

We fix some notation before stating the next result.  Let $\Vert \cdot\Vert$
denote the Euclidean norm when applied to vectors and the spectral (induced)
norm when applied to matrices, $\Vert \cdot \Vert_F$ denotes the Frobenius norm
for matrices, and superscript $+$ denotes the Moore--Penrose pseudo-inverse.
Least squares estimators of $\Alpha$ and $\alpha$ are denoted by $\hat{\Alpha} =
(Z^\tsp Q_X Z)^{+}Z^\tsp Q_X Y$ and $\hat{\alpha} = \vecop(\hat{\Alpha})$,
respectively, and $y = \vecop(Y)$.

\begin{lem} \label{lem:drift}
  Define $V:\R{rq^2} \to [0, \infty)$ by $V(\alpha) = \Vert \alpha\Vert^2$. If
  $C \in \SPD{qr^2}$ and at least one of the two sets of conditions in
  Proposition~\ref{prop:prop_post} holds, then for any $\lambda \geq 0$ and with
  \begin{align*}
    L =  \left(\Vert C^{-1}\Vert \Vert C m\Vert + \Vert C^{-1/2} \Vert \Vert
    C^{1/2}\hat{\alpha}\Vert \right)^2 + \tr(C^{-1}),
 \end{align*}
 the kernel $K_{\Alpha}$ satisfies the drift condition
 \[
  \int V(\alpha) K_{\Alpha}(\alpha', \dd \alpha) \leq \lambda V(\alpha') + L.
 \]
\end{lem}

\begin{proof}
 Assume first $Q_X = I_n$; the general case is then recovered by replacing $Z$
 and $Y$ by $Q_X Z$ and $Q_X Y$ everywhere. Using \eqref{eq:transition_M} and
 Fubini's Theorem yields
 \begin{align*}
  \int \Vert \alpha \Vert^2 K_{\Alpha}(\alpha', \dd \alpha) & =
    \iint \Vert \alpha \Vert^2 f(\alpha \mid \Sigma, Y, X)%
    f(\Sigma \mid \alpha', Y, X)\, \dd \Sigma\,\dd \alpha\\
    & = \iint \Vert \alpha \Vert^2 f(\alpha \mid \Sigma, Y,X)%
    f(\Sigma \mid \alpha', Y, X)\,\dd \alpha \, \dd \Sigma.
 \end{align*}
 Lemma \ref{lem:collapse_dist} and standard expressions for the moments of the
 multivariate normal distribution \citep[][Theorem 10.18]{Schott2005} give for
 the inner integral that
 \[
  \int \Vert \alpha\Vert^2 f(\alpha \mid \Sigma, Y, X)\,\dd \alpha = \Vert
  u\Vert^2 + \tr(B^{-1}).
 \]
 The triangle inequality gives $\Vert u \Vert \leq \Vert B^{-1}Cm\Vert + \Vert
 B^{-1}(\Sigma^{-1} \otimes Z^\tsp) y \Vert$. We work separately on the last two
 summands. First, since $\Sigma^{-1}\otimes Z^\tsp Z$ is SPSD, we get by Lemma
 \ref{lem:sum_pd_psd} that
\begin{align*}
  \Vert B^{-1}Cm\Vert \leq \Vert C^{-1}\Vert \Vert C m\Vert.
\end{align*}
Secondly,
 \begin{align*}
   & \Vert B^{-1}(\Sigma^{-1} \otimes Z^\tsp)y \Vert \\
   & = \Vert C^{-1/2}(I_{qr^2} + C^{-1/2}(\Sigma^{-1}\otimes Z^\tsp Z)%
   C^{-1/2})^{-1}C^{-1/2} (\Sigma^{-1}\otimes Z^\tsp) y \Vert \\
   & \leq \Vert C^{-1/2}\Vert \Vert (I_{qr^2} + C^{-1/2} (\Sigma^{-1}\otimes%
   Z^\tsp Z) C^{-1/2})^{-1}C^{-1/2} (\Sigma^{-1}\otimes Z^\tsp) y \Vert
 \end{align*}
Now by Lemma \ref{lem:small_coef}, with $(\Sigma^{-1/2}\otimes I_n)y$ and
$(\Sigma^{-1/2} \otimes Z) C^{-1/2}$ taking the roles of what is there denoted
$y$ and $X$, we have for any generalized inverse (denoted by superscript $g$)
that
\[
  \Vert (I_{qr^2} + C^{-1/2}(\Sigma^{-1} \otimes Z^\tsp Z)
  C^{-1/2})^{-1}C^{-1/2}(\Sigma^{-1}\otimes Z^\tsp)y \Vert
\]
is upper bounded by
\begin{equation} \label{eq:proof_eq_1}
  \Vert (C^{-1/2}(\Sigma^{-1} \otimes Z^\tsp Z) C^{-1/2})^{g}C^{-1/2}
(\Sigma^{-1}\otimes Z^\tsp)y \Vert.
\end{equation}
Lemma \ref{lem:g_inv_prod} says that $C^{1/2}(\Sigma^{-1} \otimes Z^\tsp Z)^+
 C^{1/2}$ is one such generalized inverse. Using that the Moore--Penrose
pseudo-inverse distributes over the Kronecker product
\citep{Magnus.Neudecker2002}, the middle part of this generalized inverse can be
written as $(\Sigma^{-1}\otimes Z^\tsp Z)^+ = \Sigma \otimes (Z^\tsp Z)^+$.
Thus, for this particular choice of generalized inverse \eqref{eq:proof_eq_1} is
equal to
\[
  \Vert C^{1/2}(\Sigma \otimes [Z^\tsp Z]^+)(\Sigma^{-1} \otimes Z^\tsp) y\Vert
  = \Vert C^{1/2}(I_r \otimes [Z^\tsp Z]^+ Z^\tsp)y\Vert.
\]
Thus, using also that $\tr(B^{-1}) \leq \tr(C^{-1})$ by Lemma
\ref{lem:sum_pd_psd} since $\Sigma^{-1} \otimes Z^\tsp Z$ SPSD, $\Vert u\Vert^2
+ \tr(B^{-1})$ is less than
\[
  \left(\Vert C^{-1}\Vert \Vert C m\Vert + \Vert C^{-1/2} \Vert \Vert C^{1/2}(I_r
\otimes [Z^\tsp Z]^+ Z^\tsp)y\Vert \right)^2 + \tr(C^{-1}).
\]
 Since the right-hand side does not depend on $\Sigma$, the proof is
 completed upon integrating both sides with respect to
 $f(\Sigma \mid \alpha', Y)\, \dd \Sigma$.
\end{proof}

\begin{lem} \label{lem:minor} If at least one of the two sets of
 conditions in Proposition~\ref{prop:prop_post} holds, then for any
 $T > 0$ and $\alpha$ such that $\Vert \alpha\Vert^{2} \leq T$, there
 exists a probability measure $R$ and
 \[
  \veps =  \frac{\vert D + Y^\tsp Q_{[Z, X]}Y \vert^{(n + a - p - r - 1)/2}}
  {\vert D + I_r(\Vert Q_X Y\Vert + \Vert Q_X Z\Vert \sqrt{T})^2\vert^{(n + a -
  p - r - 1)/2}} > 0
 \]
 such that
 \[
  K_{\Alpha}(\alpha,\cdot) \geq \veps R(\cdot).
 \]
\end{lem}

\begin{proof}
  We will prove that there exists a function $g:\SPD{r}\to (0, \infty)$, depending
  on the data and hyperparameters, such that $\int g(\Sigma)\,\dd \Sigma > 0$ and
  $g(\Sigma) \leq f(\Sigma\mid \Alpha, Y)$ for every $\alpha$ such that $\Vert
  \alpha \Vert^2 \leq T$, or, equivalently, $\Vert \Alpha\Vert_F^2 \leq T$. This
  suffices since if such a $g$ exists, then we may take $\veps = \int g(\Sigma)\,
  \dd \Sigma$ and define the distribution $R$ by, for any Borel set $A \subseteq
  \R{qr^2}$, \[R(A) = \frac{1}{\veps} \iint I_A(\alpha) f(\alpha \mid \Sigma, Y)
  g(\Sigma) \, \dd \alpha\, \dd \Sigma.\]

 Let $c = n + a - p - r - 1$ and $E = Y - Z\Alpha$ so that $f(\Sigma\mid \Alpha,
 Y)$ can be written
 \begin{align*}
  \frac{\vert D +  E^\tsp Q_X E\vert^{c/2}}{2^{cr/2}\Gamma_r(c/2)}\vert
  \Sigma\vert^{-\frac{n + a - p}{2}}\etr\left(-\frac{1}{2}\Sigma^{-1}[D + E^\tsp
  Q_X E]\right).
 \end{align*}
 To establish existence of a $g$ with the desired properties we will lower bound
 the first and third term in $f(\Sigma\mid \Alpha, Y)$ using two inequalities,
 namely
 \[
  \vert D + E^\tsp Q_X E\vert \geq \vert D + Y^\tsp Q_{[Z, X]}Y\vert
 \]
 and, for every $\Alpha$ such that $\Vert \Alpha\Vert_F^2 \leq T$,
 \[
  \tr\left[\Sigma^{-1}E^\tsp Q_X E\right] \leq  \tr\left[\Sigma^{-1}\left(\Vert
  Q_X Y\Vert + \Vert Q_X Z\Vert \sqrt{T} \right)^2\right].
 \]
 We prove the former inequality first. Since $E^\tsp Q_{[Z, X]}E = Y^\tsp Q_{[Z,
X]}Y$, it suffices to prove that $\vert D + E^\tsp Q_X E\vert \geq \vert D +
E^\tsp Q_{[Z, X]}E\vert$. For this, Lemma \ref{lem:sum_pd_psd}.3 says it is
enough to prove that $E^\tsp Q_X E - E^\tsp Q_{[Z, X]} E$ is SPSD. But the
Frisch--Waugh--Lovell theorem \citep[Section 2.4]{MacKinnon.Davidson2003} says
$E^\tsp Q_{[Z, X]}E = (Q_X E)^\tsp Q_{Q_X Z}(Q_X E)$, and therefore
 \begin{align*}
  E^\tsp Q_X E - E^\tsp Q_{[Z, X]} E & = (Q_X E)^\tsp (I_n - Q_{Q_X Z})Q_X E \\
  & = [(I_n - Q_{Q_X Z})Q_X E]^\tsp [(I_n - Q_{Q_X Z})Q_X E],
 \end{align*}
 which is clearly SPSD.

 For the second inequality we get, using the triangle inequality,
submultiplicativity, and that the Frobenius norm upper bounds the spectral norm,
 \begin{align*}
  \Vert E^\tsp Q_X E\Vert & = \Vert(Q_X E)^\tsp Q_X E\Vert\\
    & \leq \Vert Q_X E\Vert^2\\
    & \leq (\Vert Q_X Y\Vert + \Vert Q_X Z\Vert \Vert \Alpha\Vert)^2\\
    & \leq (\Vert Q_X Y\Vert + \Vert Q_X Z\Vert \Vert \Alpha\Vert_F)^2\\
    & \leq (\Vert Q_X Y\Vert + \Vert Q_X Z\Vert \sqrt{T})^2\\
    & = : c_1.
 \end{align*}
 Since the spectral norm for SPSD matrices is the maximum eigenvalue, we have
 shown that $c_1I_r - E^\tsp Q_X E$ is SPSD. Thus, $\Sigma^{-1/2}(I_r c_1 -
 E^\tsp Q_X E)\Sigma^{-1/2}$ is also SPSD and, hence,
 \[
  \tr(\Sigma^{-1}E^\tsp Q_X E) = \tr(\Sigma^{-1/2}E^\tsp Q_X E \Sigma^{-1/2})
  \leq  \tr(\Sigma^{-1} c_1),
 \]
 which is what we wanted to show. We have thus established that $ f(\Sigma \mid
 \Alpha, Y)$ is greater than
 \[
  g(\Sigma) := \frac{\vert D + Y^\tsp Q_{[Z, X]} Y
  \vert^{c/2}}{2^{cr/2}\Gamma_r(c/2)}\vert \Sigma\vert^{-\frac{n + a -
  p}{2}}\etr\left(-\frac{1}{2}\Sigma^{-1}[D + I_r c_1]\right).
 \]
 Finally, the stated expression for $\veps = \int g(\Sigma)\,\dd\Sigma$, and
 that it is indeed positive, follows from that under the first set of conditions
 in Proposition~\ref{prop:prop_post} $D$ is SPD, and under the second set of
 conditions $Y^\tsp Q_{[Z, X]}Y$ is SPD by Lemma \ref{lem:partition_rank}; in
 either case, both $D + Y^\tsp Q_{[Z, X]}Y$ and $D + I_r c_1$ are SPD and,
 consequently, $g$ is proportional to an inverse Wishart density with scale
 matrix $D + I_r c_1$ and $c$ degrees of freedom.
\end{proof}

We are ready for the main result of this section.

\begin{thm} \label{thm:geo_erg}
  If $C \in \SPD{qr^2}$ and at least one of the two sets of conditions in
  Proposition~\ref{prop:prop_post} holds, then the transition kernels $K_C$,
  $K_G$, and $K_{\Alpha}$ are geometrically ergodic.
\end{thm}

\begin{proof}
 By Lemma \ref{lem:tv_dist} it suffices to show it for $K_{\Alpha}$. Lemma
\ref{lem:drift} establishes that a drift condition \eqref{eq:drift_cond} holds
for $K_{\Alpha}$ with $V(\alpha) = \Vert \alpha\Vert^2$ and all $\lambda \in [0,
1)$, while Lemma~\ref{lem:minor} establishes a minorization condition
\eqref{eq:minor_cond} for $K_{\Alpha}$.  The claim now follows immediately from
Theorem~\ref{thm:Rosenthal}.
\end{proof}

We note that, since $V(\alpha)$ is unbounded off compact sets and it is routine
to show that $K_{\Alpha}$ is weak Feller \citep[][p. 124]{Meyn.Tweedie2011}, the
theorem can in fact be proven without using Lemma \ref{lem:minor} \citep[][Lemma
15.2.8]{Meyn.Tweedie2011}. However, with Lemma \ref{lem:minor} we also get an
explicit bound on the convergence rate, through Theorem \ref{thm:Rosenthal},
which will be useful in what follows.

\section{Asymptotic stability}
\label{sec:asy_stab}
We consider asymptotically stable geometric ergodicity as $n\to \infty$.
Motivated by Lemma \ref{lem:tv_dist}, we focus on the sequence of kernels
$\{K_{\Alpha, n}\}$, where $K_{\Alpha, n}$ is the kernel $K_{\Alpha}$ with the
dependence on the sample size $n$ made explicit; we continue to write
$K_{\Alpha}$ when $n$ is arbitrary but fixed. Similar notation applies to the
kernels $K_C$ and $K_G$.

It is clear that as $n$ changes so do the data $Y$ and $X$. Treating $Y$ and $X$
as fixed is not appropriate unless we only want to discuss asymptotic properties
holding pointwise, i.e. for particular, or all, paths of the stochastic process
$\{(Y_t, X_t)\}$, which is unnecessarily restrictive. We assume that $(Y_t,
X_t)$, $t = 1, 2, \dots$ are defined on a common probability space so the joint
distribution of $Y$ and $X$ exists for every $n$, and we allow for model
misspecification; that is, $\{Y_t\}$ and $\{X_t\}$ need not satisfy
\eqref{eq:varx}.

Recall that Theorem~\ref{thm:Rosenthal} is instrumental to our strategy: if
$K_{\Alpha, n}$ satisfies Theorem~\ref{thm:Rosenthal} with some $V = V_n$,
$\lambda = \lambda_n$, $L = L_n$, $\veps = \veps_n$, and $T = T_n$, then there
exists a $\bar{\rho}_n < 1$ that upper bounds $\rho^\star_n$. We focus on the
properties of those $\bar{\rho}_n$, $n = 1, 2, \dots$, as $n$ tends to infinity.
Throughout the section we assume that the priors, and in particular the
hyperparameters, are the same for every $n$.  The latter is not necessary and
could be replaced by appropriate bounds on how the hyperparameters change with
$n$; however, doing so complicates notation and does not lead to fundamental
insights in our setting. For example, $C$ could be allowed to vary with $n$ as
long as its eigenvalues are bounded away from zero and from above.

Clearly, the choice of drift function $V_n$ is important for the upper bound
$\bar{\rho}_n$ one obtains. The drift function used for the small-$n$ regime is
not well suited for the asymptotic analysis in this section. Essentially,
problems occur if $\lambda_n \to 1$, $L_n \to \infty$, or $\veps_n \to 0$ so
that the corresponding upper bounds satisfy $\lim_{n \to \infty} \bar{\rho}_n =
1$ almost surely \citep[][Proposition 2]{Qin.Hobert2019}.  Consider
Theorem~\ref{thm:geo_erg}.  Since we can take $\lambda_n = 0$ for all $n$, only
$L_n$ or $\veps_n$ can lead to problems.  Because $L_n$ is essentially a
quadratic in $\alpha$, it is clear that $L_n(Y, X) = O_\pr(1)$ if and only if
$\Vert \hat{\alpha}\Vert = O_\pr(1)$, while we show in Appendix
\ref{app:inadequacy} that $\veps_n \to 0$ almost surely as $n\to \infty$ if
\[
 n \Vert Q_X Z\Vert^2/\Vert Q_{[Z, X]} Y \Vert^2 \to \infty.
\]
We expect this to occur in many relevant configurations. Indeed, we expect the
order of $\Vert Q_X Z\Vert$ will often be at least that of $\Vert Q_{[X, Z]}
Y\Vert$. To see why, consider the case without predictors  and data generated
according to the VARX. Then $\Vert Q_{[X, Z]} Y\Vert^2 = \Vert Q_Z Y\Vert^2 = n
\emax(Y^\tsp Q_Z Y / n)$, where $\emax(\cdot)$ denotes the maximum eigenvalue,
and $Y^\tsp Q_Z Y / n$ is the maximum likelihood estimator of $\Sigma$ which is
known to be consistent, for example, if data are generated from a stable VAR
with i.i.d. Gaussian innovations \citep{Lutkepohl2005}. For such VARs it also
holds that $Z^\tsp Z / n$ converges in probability to some SPD limit
\citep{Lutkepohl2005}, and hence $\Vert Z\Vert^2 = n \emax(Z^\tsp Z / n) =
O_\pr(n)$. Similar arguments can be made for any other data generating
processes for which $Y^\tsp Q_Z Y / n$ and $Z^\tsp Z / n$ are suitably bounded
in probability or almost surely as $n\to \infty$.

The intuition as to why the drift function that works in the small-$n$ regime is
not suitable for convergence complexity analysis is that the drift function
should be centered (minimized) at a point the chain in question can be expected
to visit often \citep{Qin.Hobert2019}. The function defined by $V(\alpha) =
\Vert \alpha\Vert^2$ is minimized when $\alpha = 0$, but there is in general no
reason to believe the $\alpha$-component of the chain will visit a neighborhood
of the origin often. On the other hand, if the number of observations grows fast
enough in comparison to other quantities and we suppose momentarily that the
data are generated from the VARX, then we expect the marginal posterior density
of $\Alpha$ to concentrate around the true $\Alpha$, i.e. the $\Alpha$ according
to which the data is generated. We also expect that for large $n$ the least
squares and maximum likelihood estimator $\hat{\Alpha} = (Z^\tsp Q_X Z)^+ Z^\tsp
Q_X Y$ is close to the true $\Alpha$. Thus, intuitively, the $\alpha$-component
of the chain should visit the vicinity of $\hat{\alpha} = \vecop(\hat{\Alpha})$
often. Formalizing and extending this intuition to cases where the model can be
misspecified, so that no true $\Alpha$ exists, leads to the main result of the
section.

Let us re-define $V:\R{qr^2} \to [0, \infty)$ by $V(\alpha) = \Vert Q_X Z \Alpha
- Q_X Z \hat{\Alpha}\Vert_F^2 = \Vert (I_r \otimes Q_X Z)(\alpha -
\hat{\alpha})\Vert^2$. We will use the following lemma to verify the drift
condition in \eqref{eq:drift} for all large enough $n$ almost surely or with
probability tending to one. The probabilistic qualifications are needed because,
in contrast to in the small $n$ setting, $V$ here depends on the data.
Accordingly, the $\lambda$ given in the lemma depends on $\hat{\Alpha}$ and,
consequently, need not be less than one for a fixed $n$ or a particular sample.

\begin{lem} \label{lem:drift_asy}
 If $[Z, X]$ has full column rank, $C \in \SPD{qr^2}$, at  one of the two sets
 of conditions in Proposition~\ref{prop:prop_post} holds,
 \[
  \lambda = \frac{qr +  \left(\Vert C\Vert^{1/2}\Vert \hat{\Alpha}\Vert_F + \Vert
  C^{-1}\Vert^{1/2}\Vert C m\Vert \right)^2}{n + a - 2r - p - 2},~~ \text{and}~
  L = \lambda \tr(D) + \lambda \Vert Q_{[Z, X]}Y\Vert_F^2,
 \]
 then
 \begin{align*}
  \int V(\alpha)K_{\Alpha}(\alpha',\dd \alpha) \leq \lambda V(\alpha') + L.
 \end{align*}
\end{lem}

\begin{proof}
 Suppose first that $Q_X = I_n$ and notice that $Z$ has full column rank, and
hence $(Z^\tsp Z)^{-1}$ exists. Since $f(\alpha \mid \Sigma, Y, X)$ is a
multivariate normal density, standard expressions for the moments of the
multivariate normal distribution gives
 \begin{align} \label{eq:drift}
  \int V(\alpha, \Sigma)f(\alpha\mid \Sigma, Y, X)\,\dd \alpha = \Vert (I_r
  \otimes Z)(u - \hat{\alpha})\Vert^2 + \tr((I_r \otimes Z) B^{-1}(I_r \otimes
  Z)^\tsp).
 \end{align}
 For the second term we use cyclical invariance of the trace to write
 \begin{align*}
   & \tr\left((I_r \otimes Z) B^{-1}(I_r \otimes Z)^\tsp\right)\\
   & = \tr\left[(I_r \otimes Z^\tsp Z)
       (C + \Sigma^{-1}\otimes Z^\tsp Z)^{-1}\right]\\
   & = \tr\left[(I_r \otimes Z^\tsp Z)^{1/2}(C + \Sigma^{-1}\otimes
       Z^\tsp Z)^{-1}(I_r \otimes Z^\tsp Z)^{1/2}\right].
 \end{align*}
 Since $C$ and $\Sigma^{-1}\otimes Z^\tsp Z$ are both SPD, the last expression
 is in the form required by Lemma \ref{lem:sum_pd_psd}, and hence
 \begin{align*}
  \tr\left((I_r \otimes Z) B^{-1}(I_r \otimes Z)^\tsp\right)
   & \leq \tr\left[(I_r \otimes Z^\tsp Z)^{1/2}(\Sigma^{-1}\otimes
      Z^\tsp Z)^{-1}(I_r \otimes Z^\tsp Z)^{1/2} \right] \\
   & = \tr\left[\Sigma \otimes Z(Z^\tsp Z)^{-1}Z^\tsp\right]\\
   & = \tr(\Sigma)\tr[Z(Z^\tsp Z)^{-1}Z^\tsp]\\
   & = \tr(\Sigma)qr,
 \end{align*}
 where the last line uses that the trace of a projection matrix is the dimension
 of the space onto which it is projecting. Focusing now on the first term on the
 right hand side in \eqref{eq:drift} we have, defining $H = \Sigma^{-1} \otimes
 Z^\tsp Z$ and using $\hat{\alpha} = H^{-1}(\Sigma^{-1}\otimes Z^\tsp)y$, that
 \[
  \Vert(I_r \otimes Z)(u - \hat{\alpha})\Vert = \Vert (I_r \otimes
  Z)(\hat{\alpha} - B^{-1}(C m + [\Sigma^{-1}\otimes Z^\tsp]y))\Vert
 \]
 is upper bounded by
 \begin{align} \label{eq:drift_2}
  \Vert (I_r \otimes Z)(H^{-1} - B^{-1})(\Sigma^{-1}\otimes Z^\tsp)y\Vert +
  \Vert (I_r \otimes Z) B^{-1}Cm\Vert.
 \end{align}
 Moreover, since $B = C + H$ the Woodbury identity gives $H^{-1} - B^{-1} =
H^{-1}(C^{-1} + H^{-1})^{-1}H^{-1}$ so that the first term in \eqref{eq:drift_2}
can be upper bounded as follows:
 \begin{align*}
   & \Vert (I_r \otimes Z)(H^{-1} - B^{-1})(\Sigma^{-1}\otimes Z^\tsp)y\Vert\\
   & =\Vert (I_r \otimes Z) H^{-1}(H^{-1} + C^{-1})^{-1}H^{-1}(\Sigma^{-1}
      \otimes Z^\tsp)y\Vert\\
   & = \Vert (I_r \otimes Z) H^{-1}(H^{-1} + C^{-1})^{-1}\hat{\alpha}\Vert\\
   & \leq \Vert (I_r \otimes Z) H^{-1/2}\Vert \Vert H^{-1/2}(H^{-1} +
    C^{-1})^{-1} \Vert \Vert \hat{\alpha}\Vert.
 \end{align*}
 Here, the power $G^t$, $t \in \R{}$, for a SPD matrix $G$ is defined by taking
the spectral decomposition $G = U_G \diag(\emax(G), \dots, \emin(G))U_G^\tsp$,
where $\emax(\cdot)$ and $\emin(\cdot)$ denote the largest and smallest
eigenvalues, respectively, and setting
\[
  G^t = U_G\diag(\emax^t(G), \dots, \emin^t(G))U_G^\tsp.
\]
Now by standard properties of eigenvalues and
eigenvectors of Kronecker products \citep[][Theorem 4.2.12]{Horn.Johnson1991} we
get
 \[
  \Vert (I_r \otimes Z)H^{-1/2}\Vert = \Vert\Sigma^{1/2}\otimes Z [Z^\tsp
  Z]^{-1/2}\Vert = \Vert \Sigma^{1/2}\Vert \Vert Z(Z^\tsp Z)^{-1/2}\Vert = \Vert
  \Sigma^{1/2}\Vert.
 \]
 In addition, using Lemma \ref{lem:sum_pd_psd},
 \begin{align*}
  \Vert H^{-1/2}(H^{-1} + C^{-1})^{-1}\Vert & = \emax^{1/2}\left((H^{-1} +
  C^{-1})^{-1}H^{-1}(H^{-1} + C^{-1})^{-1}\right) \\
  & \leq \emax^{1/2}\left((H^{-1} + C^{-1})^{-1}\right)\\
  & \leq \emax^{1/2}(C)\\
  & = \Vert C\Vert^{1/2}.
 \end{align*}
 It remains to deal with the second term in \eqref{eq:drift_2}. Using a similar
 technique as with the previous term, applying sub-multiplicativity and Lemma
 \ref{lem:sum_pd_psd} twice, we have
 \begin{align*}
   & \Vert (I_r \otimes Z)B^{-1}Cm\Vert\\
   & = \Vert(\Sigma^{1/2}\otimes I_n)(\Sigma^{-1/2} \otimes I_n)(I_r \otimes Z)
   B^{-1}Cm\Vert\\
   & \leq \Vert \Sigma^{1/2}\Vert \Vert (\Sigma^{-1/2} \otimes Z)
     (C + \Sigma^{-1}\otimes Z^\tsp Z)^{-1}\Vert \Vert Cm\Vert\\
   & = \Vert \Sigma^{1/2}\Vert \emax^{1/2}\left([C + \Sigma^{-1}\otimes
       Z^\tsp Z]^{-1}[\Sigma^{-1}\otimes Z^\tsp Z]
       [C + \Sigma^{-1}\otimes Z^\tsp Z]^{-1} \right)\Vert Cm \Vert \\
   & \leq \Vert \Sigma^{1/2}\Vert \emax^{1/2}([C + \Sigma^{-1}\otimes
     Z^\tsp Z]^{-1})\Vert Cm\Vert\\
   & \leq \Vert \Sigma^{1/2}\Vert \Vert C^{-1} \Vert^{1/2} \Vert C m\Vert.
 \end{align*}
 Putting things together we have shown that, for any $\Sigma$,
 \begin{align*}
  \Vert (I_r \otimes Z)(\hat{\alpha} - u)\Vert \leq \Vert
  \Sigma^{1/2}\Vert\left( \Vert C\Vert^{1/2}\Vert \hat{\alpha}\Vert + \Vert
  C^{-1}\Vert^{1/2} \Vert C m\Vert\right),
 \end{align*}
 and hence we get from \eqref{eq:drift}
 \begin{align*}
  \int V(\alpha)f(\alpha \mid \Sigma, Y, X)\, \dd \alpha
    & \leq \Vert \Sigma\Vert\left(\Vert C\Vert^{1/2}\Vert \hat{\alpha}\Vert +
    \Vert C^{-1}\Vert^{1/2} \Vert C m\Vert \right)^2 + qr \tr(\Sigma).
 \end{align*}
 The proof for the case $Q_X = I_n$ is completed by upper bounding $\Vert
 \Sigma\Vert \leq \tr(\Sigma)$, integrating both sides with respect to
 $f(\Sigma\mid \alpha', Y, X)\,\dd\Sigma$, and noting that
 \begin{align*}
   & \int \tr(\Sigma)f(\Sigma\mid \alpha', Y, X)\,\dd\Sigma\\
   & = \frac{1}{n + a - 2r - p - 2}\tr\left(D + (Y - Z\Alpha')^\tsp
    (Y - Z\Alpha')\right)\\
   & \leq \frac{1}{n + a - 2r - p - 2}\left(\tr(D) + \Vert Q_ZY\Vert^2_F +
      \Vert Z\hat{\Alpha} - Z\Alpha'\Vert_F^2\right),
 \end{align*}
 where we have used that $(Y - Z \Alpha')^\tsp(Y - Z\Alpha') = (Y - Z
 \Alpha')^\tsp P_Z (Y - Z\Alpha') + (Y - Z \Alpha')^\tsp Q_Z(Y - Z\Alpha')$, and
 that $P_ZY = Z\hat{\Alpha}$. The general case is recovered by replacing $Z$ and
 $Y$ by $Q_X Z$ and $Q_X Y$, respectively, and invoking Lemma
 \ref{lem:partition_rank}. That $Z^\tsp Q_X Z$ is invertible also in the general
 case follows from the same lemma.
\end{proof}

\begin{lem}\label{lem:minor_asy}
 If at least one of the two sets of conditions in
 Proposition~\ref{prop:prop_post} holds, then for any $T > 0$ and $\alpha =
 \vecop(\Alpha)$ such that $\Vert Q_X Z\hat{\Alpha} - Q_X Z\Alpha\Vert_F^2 \leq
 T$, there exists a probability measure $R$ and
 \[
  \veps = \left(\frac{\vert D + Y^\tsp Q_{[Z, X]} Y\vert}{\vert D +
  Y^\tsp Q_{[Z, X]} Y + I_{r}T\vert}\right)^{(n + a - p - r - 1)/2} > 0
 \]
 such that
 \[
  K_{\Alpha}(\alpha, \cdot) \geq \veps R(\cdot).
 \]
\end{lem}

\begin{proof}
 The proof idea is similar to that for Lemma \ref{lem:minor}. We prove that
 there exists a $g:\SPD{r} \to [0, \infty)$, depending on the data and the
 hyperparameters, such that $\int g(\Sigma)\,\dd\Sigma = \veps > 0$ and
 $g(\Sigma) \leq f(\Sigma\mid \Alpha, Y, X)$ for every $\Alpha$ such that $\Vert
 Q_X Z\hat{\Alpha} - Q_X Z\Alpha\Vert_F^2 \leq T$.

 Assume first that $Q_X = I_n$ and let $c = n + a - r - p - 1$ be the degrees of
 freedom in the full conditional distribution for $\Sigma$. Using that $(Y -
 Z\Alpha)^\tsp (Y - Z\Alpha) - (Y - Z\Alpha)^\tsp Q_Z (Y - Z\Alpha)$ is SPSD and
 that $Q_Z Z = 0$, we get by Lemma \ref{lem:sum_pd_psd} that
 \[
  \vert D + (Y - Z\Alpha)^\tsp(Y - Z\Alpha)\vert \geq \vert D + Y^\tsp Q_Z
  Y\vert.
 \]
 Moreover, for any $\Alpha$ such that $\Vert Z\hat{\Alpha} - Z\Alpha\Vert_F^2
 \leq T$,
 \begin{align*}
   & \tr\left[\Sigma^{-1}(Y - Z\Alpha)^\tsp (Y - Z\Alpha)\right]\\
   & = \tr\left[\Sigma^{-1}Y^\tsp Q_Z Y + \Sigma^{-1 }(Y - Z\Alpha)^\tsp P_Z
       (Y - Z\Alpha)\right]\\
   & = \tr\left[\Sigma^{-1}Y^\tsp Q_Z Y + \Sigma^{-1}(Z \hat{\Alpha} -
       Z \Alpha)^\tsp (Z\hat{\Alpha} - Z\Alpha)\right]\\
   & \leq \tr\left[\Sigma^{-1}Y^\tsp Q_Z Y + \Sigma^{-1} \Vert Z\hat{\Alpha}
      -Z\Alpha\Vert^2\right]\\
   & \leq \tr\left[\Sigma^{-1}Y^\tsp Q_Z Y + \Sigma^{-1}T\right],
 \end{align*}
 where the first inequality follows from that $\Vert Z\hat{\Alpha} -
 Z\Alpha\Vert^2 = \Vert (Z\hat{\Alpha} - Z\Alpha)^\tsp ( Z\hat{\Alpha} -
 Z\Alpha)\Vert$ and that, therefore, $I_r \Vert Z\hat{\Alpha} - Z\Alpha\Vert^2 -
 (Z\hat{\Alpha} - Z\Alpha)^\tsp (Z\hat{\Alpha} - Z\Alpha)$ is SPSD, and the
 second inequality follows from that the Frobenius norm upper bounds the
 spectral norm, so that $\Vert Z\hat{\Alpha} - Z\Alpha^2 \Vert \leq \Vert
 Z\hat{\Alpha} - Z\Alpha\Vert_F^2 \leq T$.

 With the determinant and trace inequalities just established, we have that
$f(\Sigma\mid \Alpha, Y, X)$ is, for any $\Alpha$ satisfying the hypotheses,
lower bounded by
 \begin{align*}
  g(\Sigma):=\frac{\vert D + Y^\tsp Q_Z Y\vert^{c/2}}{2^{c r/2}\Gamma_r(c
  /2)}\vert \Sigma\vert^{-\frac{n + a -
  p}{2}}\etr\left(-\frac{1}{2}\Sigma^{-1}[D + Y^\tsp Q_Z Y + I_r T]\right).
 \end{align*}
 Noticing that $g$ so defined is proportional to an inverse Wishart density and
 using well known expression for its normalizing constant finishes the proof for
 the case where $Q_X = I_n$. The general case is recovered upon replacing $Z$
 and $Y$ by $Q_X Z$ and $Q_Z Y$ everywhere and invoking Lemma
 \ref{lem:partition_rank}.
\end{proof}

We are ready to state the main result of the section. Recall, $\emax(\cdot)$
denotes the largest eigenvalue of its argument matrix; let $\emin(\cdot)$
denote the smallest.

\begin{thm} \label{thm:asy_stab}
 If
 \begin{enumerate}
  \item [(a)] $C \in \SPD{qr^2}$,
  \item [(b)] there exists a constant $M > 0$ such that, with $W = [Y, Z, X] \in
  \R{n \times (r + qr + p)}$ and $S_W = W^\tsp W / n$, almost surely as $n\to
  \infty$,
  \[
    M^{-1} \leq \liminf_{n \to \infty}\emin(S_W) \leq \limsup_{n \to
    \infty}\emax(S_W) \leq M,
  \]
 \end{enumerate}
 then $\{K_{C, n}\}, \{K_{G, n}\}$, and $\{K_{\Alpha, n}\}$ are asymptotically
 geometrically ergodic almost surely.
\end{thm}

\begin{proof}
 By Lemma \ref{lem:tv_dist}, it is enough to prove that $\limsup_{n \to
 \infty}\bar{\rho}_n < 1$ almost surely for the $\bar{\rho}_n$ corresponding to
 $K_{\Alpha, n}$. Inspecting the definition of $\bar{\rho}_n$ in Theorem
 \ref{thm:Rosenthal} one sees that it suffices to show that Lemmas
 \ref{lem:drift_asy} and \ref{lem:minor_asy} apply and that the $\lambda =
 \lambda_n$, $L= L_n$, $T = T_n$, and $\veps = \veps_n$ they give almost surely
 satisfy, respectively: (i) $\limsup_{n \to \infty} \lambda_n < 1$, (ii)
 $\limsup_{n \to \infty} L_n < \infty$, (iii) $\limsup_{n \to \infty} T_n <
 \infty$, and (iv) $\liminf_{n \to \infty} \veps_n  > 0$. The prior $f(\alpha)$
 is bounded since $C$ is positive definite by (a), and (b) implies $[Y, Z, X]$
 has full column rank almost surely for all large enough $n$, so Lemma
 \ref{lem:drift_asy} applies for all large enough $n$ almost surely. Moreover,
 the Frisch--Waugh--Lovell theorem \citep[Section 2.4]{MacKinnon.Davidson2003}
 says $\hat{\Alpha}$ is the upper $qr\times r$ block in the least squares
 coefficient estimate in the regression of $Y$ on $[Z, X]$, i.e. $([Z, X]^\tsp
 [Z, X])^{-1}[Z, X]^\tsp Y$. Hence, with probability tending to one,
 \begin{align*}
  \Vert \hat{\Alpha}\Vert_F & \leq \Vert ([Z, X]^\tsp [Z, X])^{-1}[Z, X]^\tsp
    Y \Vert_F\\
  & \leq  \sqrt{r} \Vert ([Z, X]^\tsp [Z, X])^{-1} [Z, X]^\tsp
     \Vert \Vert Y \Vert \\
  & \leq  \sqrt{r} M,
 \end{align*}
  which follows from upper bounding the Frobenius norm by the spectral norm times
  $\sqrt{r}$ and using the Minimax Principle \citep[Corollary
  III.1.2]{Bhatia2012}; in particular, since $[Z, X]^\tsp [Z, X] / n$ and $Y^\tsp
  Y / n$ are the trailing and leading block of $S_W$, respectively, their
  eigenvalues must be bounded between the smallest and largest of $S_W$, and so
  also between $M^{-1}$ and $M$ almost surely as $n\to \infty$. Since we have
  shown $\Vert \hat{\Alpha}\Vert_F = O(1)$, it follows that $\lambda_n = O(1/n)$,
  so (i) holds. That $\lambda_n = O(1/n)$ and $\Vert Q_{[X, Z]}Y\Vert_F^2 =
  \tr(Y^\tsp Q_{[X, Z]}Y) \leq \sqrt{r}\Vert Y^\tsp Y\Vert \leq \sqrt{r}n M$ give
  $L_n = O(1)$, i.e. (ii) holds, and hence we can pick a sequence $T_n > 2L_n(1 -
  \lambda_n)$, $n = 1, 2, \dots$, such that (iii) holds. For (iv), we have with
  $\tau_j$ denoting the $j$th eigenvalue of $D / n + Y^\tsp Q_{[X, Z]} Y / n$,
 \[
  \veps_n = \prod_{j = 1}^r \left(\frac{\tau_j + T_n / n}{\tau_j}\right)^{-(n + a
- p - r - 1)/2}.
 \]
 Now, we have picked $T_n$ so that $\bar{T} = \limsup_{n\to \infty} T_n <
 \infty$, and since $Y^\tsp Q_{[Z, X]}Y$ is the Schur complement of $[Z, X]^\tsp
 [Z, X]$ in $W^\tsp W$, its eigenvalues are bounded between $n M^{-1}$ and $n M$
 almost surely as $n\to \infty$ \citep[Theorem 5]{Smith1992}. Thus,
 $\tau_j > M^{-1}/2$ (say) for all large enough $n$ almost surely, for every
 $j$. Thus, almost surely,
 \[
  \liminf_{n\to \infty} \veps_n \geq \lim_{n\to \infty} \left(\frac{M^{-1}/2 +
  \bar{T} / n}{M^{-1}/2}\right)^{- r (n + a - p - r - 1)/2} > 0,
 \]
 where the final inequality follows from that the fraction in parentheses is $1
 + (2^{-1}M^{-1}\bar{T}) / n$ and that the exponent is of the same order as $n$;
 this finishes the proof.
\end{proof}

If assumption (b) is relaxed to holding with probability tending to one instead
of almost surely, then the conclusion can be weakened accordingly to give the
following corollary. We omit the proof since it is essentially the same as the
proof of Theorem \ref{thm:asy_stab}, arguing that the necessary conditions hold
with probability tending to one.

\begin{corol}
 If $C\in \SPD{qr^2}$ and there exists an $M > 0$ such that $M^{-1} \leq
 \emin(S_W) \leq \emax(S_W) \leq M$ with probability tending to one as $n\to
 \infty$, then $\{K_{C, n}\}, \{K_{G, n}\}$, and $\{K_{\Alpha, n}\}$ are
 asymptotically geometrically ergodic in probability.
\end{corol}

\subsection{Example}

We illustrate the behavior of $\lambda = \lambda_n$, $L = L_n$, and $\veps =
\veps_n$ from Lemmas \ref{lem:drift_asy} and \ref{lem:minor_asy} in a stable
VARX of order $q = 1$; that is $\Alpha = \Alpha_1$ and $\Vert \Alpha\Vert < 1$
in \eqref{eq:varx}. Stability makes some the arguments in this example easy to
motivate formally, but we emphasize that it is not needed for the preceding
results. We take $X_t = 1$ for all $t$, $\Beta = 1_r^\tsp$, $\Sigma = \sigma^2
I_r$, and pick hyperparameters $C = I_{r^2}$, $m = 0$, $D = 0$, and $a = 0$. We
first examine $\lambda_n$, $L_n$, and $\veps_n$ analytically and then illustrate
those calculations using simulations.

In the present setting, the expressions for $\lambda_n$ and $L_n$ in Lemma
\ref{lem:drift_asy} simplify to
\[
 \lambda_n = \frac{r + \Vert \hat{\Alpha}\Vert_F^2}{n - 2r - 3} \quad
 \text{and}\quad L_n  = \frac{r + \Vert \hat{\Alpha}\Vert_F^2}{n - 2r - 3}
 \tr(Y^\tsp Q_{[Z, X]} Y).
\]
Because the VARX is stable, $\hat{\Alpha}$ and $Y^\tsp Q_{[Z, X]}Y / n$ are
consistent for $\Alpha$ and $\Sigma = \sigma^2 I_r$, respectively, as $n\to
\infty$ \citep{Lutkepohl2005}, which implies
\[
 \lambda_n = \frac{r + \Vert \Alpha\Vert_F^2}{n - 2r - 3} + o_\pr(n^{-1})
\]
and
\[
 L_n  = \frac{r + \Vert \Alpha\Vert_F^2}{n - 2r - 3} \tr(n\Sigma) + o_\pr(1) =
 r\sigma^2 (r + \Vert \Alpha\Vert_F^2) + o_\pr(1).
\]
Clearly, $\lambda_n = O_\pr(1/n)$ and $L_n = O_\pr(1)$. To get some intuition
for how $r$ affects $\lambda_n$ and $L_n$, let us ignore the stochastic terms
that are of lower order when $n$ grows with $r$ fixed; that is, consider the
approximations
\[
 \lambda_n \approx \tilde{\lambda}_n = \frac{r + \Vert \Alpha\Vert_F^2}{n - 2r -
 3} ~~ \text{and} ~~ L_n\approx \tilde{L}_n = r\sigma^2 (r + \Vert
 \Alpha\Vert_F^2).
\]
Because $\Vert \Alpha\Vert_F^2 \leq r \Vert \Alpha\Vert^2 \leq r$, it holds that
$\tilde{\lambda}_n < 1$ for large enough $n$ even if $r$ grows, as long as $r =
o(n)$. By contrast, $\tilde{L}_n$ is of  order $r^2$ if $r$ grows, regardless of
$n$. This suggests, at least informally, that $\lambda_n$ can stay below one if
$r$ grows with $n$ but that $L_n$ may not be bounded in such settings.

To investigate how $\veps_n$ behaves, note that we can take $T_n = \tilde{T}_n +
o_\pr(1)$ with $\tilde{T}_n = 2\tilde{L}_n =  2 r \sigma^2 (r + \Vert
\Alpha\Vert_F^2)$ and have $T_n > 2 L_n / (1 - \lambda_n)$ with probability
tending to one as $n$ grows with $r$ fixed. Thus, using that $Y^\tsp Q_{[Z,
X]}Y/n \to \Sigma = \sigma^2 I_r$ in probability and that the determinant is a
continuous mapping,
\begin{align*}
 \veps_n & = \left(\frac{\vert Y^\tsp Q_{[Z, X]}Y / n\vert}{\vert Y^\tsp Q_{[Z,
 X]}Y / n + I_r(\tilde{T}_n / n + o_\pr(n^{-1}))\vert}\right)^{(n - r - 2) / 2}\\
 & = \left(\frac{\sigma^2}{\sigma^2 + \tilde{T}_n / n}\right)^{r(n - r - 2) / 2}
     + o_\pr(1)\\
 & = \exp[-r^2(r + \Vert \Alpha\Vert_F^2)] + o_\pr(1).
\end{align*}
Consider the approximation $\veps_n \approx \tilde{\veps}_n = \exp[-r^2(r +
\Vert \Alpha\Vert_F^2)]$ and note, as argued previously, $\Vert \Alpha\Vert_F^2
\leq r$. Thus, if $r$ grows, then $\tilde{\veps}_n$ is of the order
$\exp(-r^3)$, regardless of $n$, and therefore we do not expect $\veps_n$ to
behave well if $r$ grows with $n$.

To investigate the finite sample behavior, we generate one sample path from the
VARX and compute $\lambda_n$, $L_n$, and $\veps_n$ using the $n$ first
observations in that sample path, for different values of $n$. We compare these
quantities to the corresponding ones from Lemmas \ref{lem:drift} and
\ref{lem:minor}, that is, those that are obtained in the small-$n$ setting. For
simplicity we also refer to the latter quantities as obtained using a
non-centered drift function. Code producing the results is available at
\url{https://github.com/koekvall/gibbs-bvarx}.

When generating data, we fix $r = 10$ and $\sigma^2 = 1$, so $\Sigma = I_{10}$.
We construct $\Alpha$ by letting $U\in \R{r \times r}$ have entries drawn
independently from the uniform distribution on $[-1/2, 1/2]$ and taking
\begin{equation*}
 \Alpha = (U + U^\tsp + I_r) / (\Vert U + U^\tsp + I_r \Vert + 0.1),
\end{equation*}
which ensures $\Vert \Alpha \Vert < 1$.

The first plot in Figure \ref{fig:sims_1} shows the $\lambda_n$ calculated using
the centered drift function from the large-$n$ setting, is less than $1$ when
$n$ is greater than 40 in our sample, but greater than one for smaller samples;
this illustrates the fact that our large-$n$ results do not control the
convergence rate in small (fixed) samples.

The second plot in Figure \ref{fig:sims_1} shows the observed values of $L_n$,
calculated using the centered drift function, appear to tend to their
probability limit
\[
 r\sigma^2(r + \Vert \Alpha\Vert_F^2) \approx 10 (10 + 2.6) = 126,
\]
where $\Vert \Alpha\Vert_F^2 \approx 2.6$ for the $\Alpha$ used to generate our
sample path. Figure \ref{fig:sims_1} also shows the behavior of $\lambda_n$ and
$L_n$ calculated using the centered drift function is, as $n$ changes, similar
to that observed if $\hat{\Alpha}$ and $Y^\tsp Q_{[Z, X]}Y / n$ are replaced by
$\Alpha$ and $\Sigma$. We have not plotted the $\lambda_n$ calculated using the
non-centered drift-function since it can be identically zero. However, the $L_n$
obtained from that drift function is plotted and is smaller (better) than that
for the centered drift function for every $n$ considered (second plot, Figure
\ref{fig:sims_1}).

The first plot in Figure \ref{fig:sims_2} indicates $\veps_n$ calculated using
the centered drift function decreases towards $\tilde{\veps}_n \approx
\exp[-100(10 + 2.6)] = \exp(-1260)$ as $n$ increases, and that there is decent
agreement with the approximation obtained by replacing $\hat{\Alpha}$ and
$Y^\tsp Q_{[Z, X]}Y / n$ by $\Alpha$ and $\Sigma$. In comparison to $\lambda_n$
and $L_n$, it seems a much larger $n$ is needed for $\veps_n$ to get close to
its probability limit. Note also that, unlike for $\lambda_n$ and $L_n$, larger
values of $\veps_n$ correspond to better convergence rate bounds.

The second plot in Figure \ref{fig:sims_2} shows the $\veps_n$ from Lemma
\ref{lem:minor}, calculated using the non-centered drift function. That
$\veps_n$ appears to tend to zero at an exponential rate (linear for its
logarithm). This rate is straightforward to verify analytically by noting that,
in the expression for $\veps_n$ in Lemma \ref{lem:minor}, the exponent is of
order $n$ and what is inside the exponent tends in probability to a number
between zero and one in the current setting. Hence, unlike the $\veps_n$ from
Lemma \ref{lem:minor_asy}, it does not provide any control over the convergence
rate asymptotically.

In summary, this example shows the drift parameters $\lambda_n$ and $L_n$ are
typically smaller (better) for the non-centered drift function used in the
small-$n$ setting, but the minorization parameter $\veps_n$ obtained using that
drift function tends to zero as $n$ tends to infinity.

\begin{figure}[!htb]
 \caption{Drift parameters $\lambda$ and $L$ with data from a stable vector
 autoregression}\label{fig:sims_1}
 \centering
 \includegraphics[width = \textwidth]{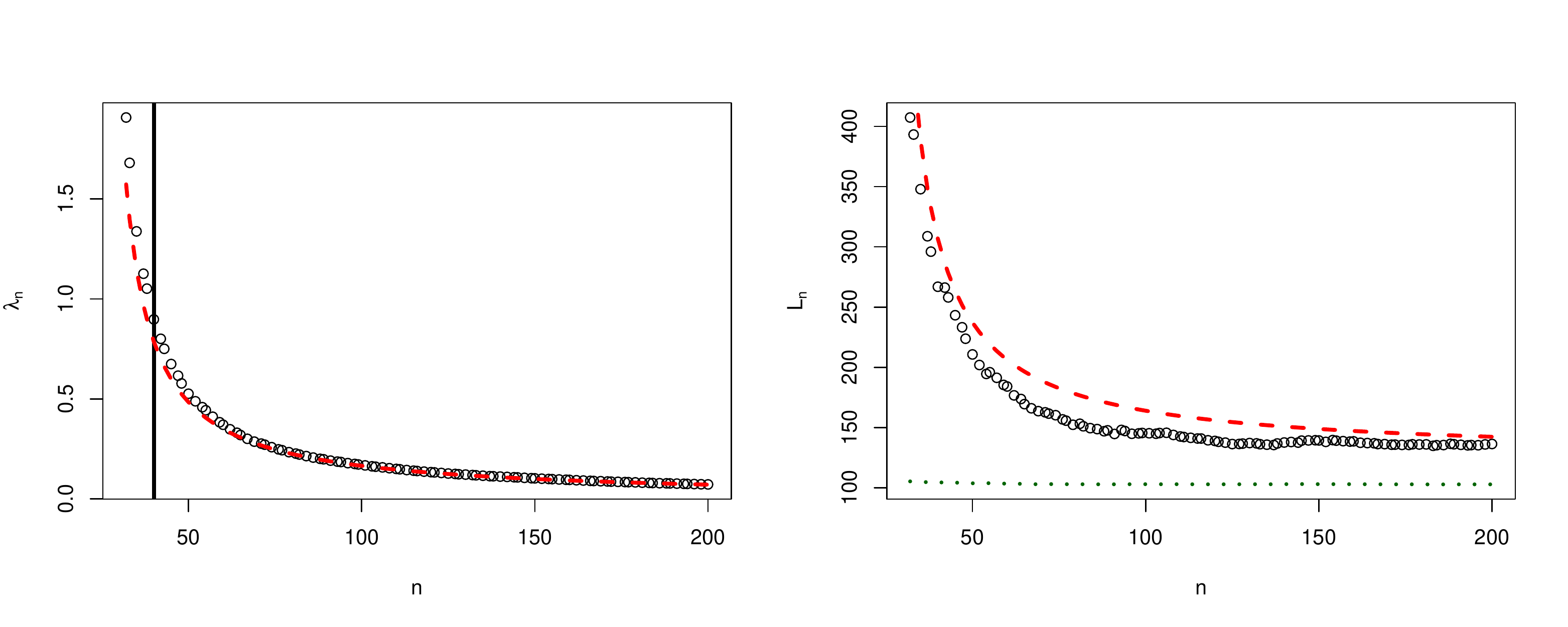}
 \floatfoot{NOTE: Circles correspond to observed $\lambda_n$ and $L_n$
 calculated as in Lemma \ref{lem:drift_asy}. Dashed lines are those $\lambda_n$
 and $L_n$ but with $\hat{\Alpha}$ and $Y^\tsp Q_{[Z, X]}Y / n$ replaced by
 $\Alpha$ and $\Sigma$. The vertical line indicates $n = 40$, the smallest $n$
 for which $\lambda_n < 1$. The dotted line is $L_n$ calculated as in Lemma
 \ref{lem:drift}; the $\lambda_n$ from that lemma is not shown because it can be
 taken to be $0$ for all $n$.}.
\end{figure}

\begin{figure}[!htb]
 \caption{Minorization parameter $\veps$ with data from a stable vector
 autoregression} \label{fig:sims_2}
 \centering
 \includegraphics[width = \textwidth]{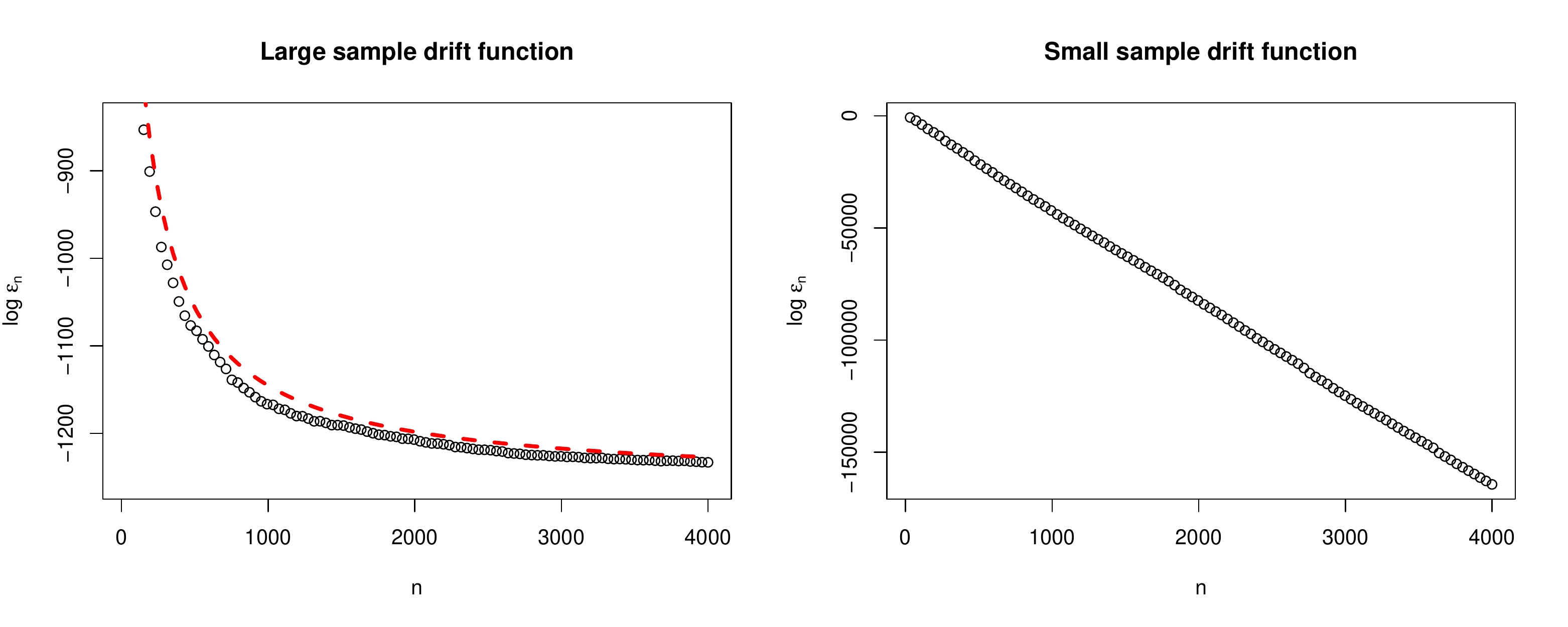}
 \floatfoot{NOTE: Circles correspond to observed $\log \veps_n$, in the left
 plot calculated as in Lemma \ref{lem:minor_asy}, with $T_n = 2 L_n (1 -
 \lambda_n) + 10^{-6}$, and in the right plot as in Lemma \ref{lem:minor} with
 $T_n = L_n + 10^{-6}$. The dashed line is $\log \veps_n$ with $\hat{\Alpha}$
 replaced by $\Alpha$ and $Y^\tsp Q_{[Z, X]}Y / n$ replaced by $\Sigma$.}.
\end{figure}

\section{Discussion} \label{sec:var_disc}

Markov chain Monte Carlo is used in a wide range of problems, including but not
limited to the Bayesian settings considered here. However, the theoretical
properties of algorithms used by practitioners are not always well understood.
We have focused on the case of Bayesian vector autoregressions with predictors.
This is one of the most common models in time series, and in particular in the
analysis and forecasting of macroeconomic time series. Moreover, due to
similarities of the likelihoods of vector autoregressions and multivariate
linear models, our results apply also to the latter. The Gibbs sampler has been
suggested for exploring the posterior distribution of the parameters $\Alpha$
and $\Sigma$ when there are no predictors \citep{Karlsson2013}, but there has
been a lack of theoretical support. We have addressed this by proposing a
collapsed Gibbs sampler that handles predictors and studying its convergence
properties. Since our algorithm simplifies to the usual Gibbs sampler when there
are no predictors, our results apply also in that setting.

We have proven that our algorithm generates a geometrically ergodic Markov chain
under reasonable assumptions (Theorem \ref{thm:geo_erg}). This result is
applicable both in classical settings where the sample size is large (but fixed)
in comparison to the number of parameters, and in large VARXs where the
dimension of the process or the lag length is large in comparison to the number
of observations. Though we have not emphasized it, this geometric
ergodicity holds also if the model is misspecified. Indeed, once the data are
fixed and the posterior specified, it is of no importance for the convergence
rate how the data were actually generated, as long as they satisfy the
conditions laid out in the relevant lemmas and theorem. Thus, with the
algorithm we consider, characteristics of the posterior distribution can be
reasonably estimated using principled approaches to ensuring the simulation
results are trustworthy \citep{Flegal.etal2008,Jones.Hobert2001,Vats.etal2019}.
Our asymptotic analysis, or convergence complexity analysis, indicates our
algorithm should perform well also in large samples; we have proven that, as the
sample size tends to infinity, the geometric ergodicity of the sequence of
transition kernels corresponding to our algorithm is asymptotically stable. This
result is one of the first of its kind for practically relevant MCMC algorithms.
As with our results for small samples, our asymptotic results hold under
almost arbitrary model mispecification, as specified in Theorem
\ref{thm:asy_stab}.

Avenues for future research include convergence complexity analysis of cases
where the dimension of the process or the lag length tends to infinity, either
together with the sample size or for a fixed sample size. Our proof uses the
intuition that the posterior concentrates near the least squares estimator, and
it may be possible to formalize this intuition also in settings where $r$, $p$,
or $q$ grows with $n$. However, it is likely that the drift function would need
to be adjusted, perhaps by using a different norm: the Frobenius norm has
convenient properties that we used in our proofs but is often less useful in
high-dimensional settings \citep{Tropp2015}. If the sample size is fixed or
grows slowly in comparison to other quantities, one would likely have to use a
different drift function altogether or move to an approach that avoids the use
of the minorization condition \citep{Qin.Hobert2019a, Qin.Hobert2019b}.

\subsection*{\centering Acknowledgements}
The authors thank an Associate Editor, two reviewers, and Dootika Vats for
suggestions and insightful comments that improved the paper. Ekvall gratefully
acknowledges support by the FWF (Austrian Science Fund,
\url{https://www.fwf.ac.at/en/}) [P30690-N35].

\bibliographystyle{abbrv}
\bibliography{var_gibbs.bib}
\appendix

\section{Preliminaries} \label{app:prel}

\begin{lem}\label{lem:partition_rank}
 If $X_i \in \R{n \times m_i}$, $m_i \in \{1, 2, \dots\}$, $i = 1, 2, 3$, and $X
 = [X_1, X_2, X_3]$ $\in \R{n \times (m_1 + m_2 + m_3)}$ has full column rank,
 then with $\tilde{X}_i = Q_{X_2}X_i$, $i = 1, 2, 3$,
 \begin{enumerate}
  \item $X_1^\tsp Q_{X_2} X_1$ is invertible,
  \item $X_1^\tsp Q_{[X_2, X_3]}X_1$ is invertible, and
  \item $X_1 ^\tsp Q_{[X_2, X_3]} X_1 = \tilde{X}_1^\tsp Q_{\tilde{X}_3}
        \tilde{X}_1$.
 \end{enumerate}
\end{lem}

\begin{proof}
 We start with 1. Suppose for contradiction that there exists $v \in
 \R{m_1}\setminus\{0\}$ such that $X_1^\tsp Q_{X_2}X_1 v = 0$, which is
 equivalent to $Q_{X_2}X_1 v = 0$. This can happen either if $X_1 v = 0$, which
 contradicts the full column rank of $X$, or if $w = X_1 v$ is a non-zero vector
 in the column space of $X_1$ that also lies in the column space of $X_2$, which
 again contradicts the full column rank of $X$. The proof for 2 is exactly the
 same as that of 1 but with $[X_2, X_3]$ in place of $X_2$. Point 3 is an
 immediate consequence the Frisch--Waugh--Lovell theorem \citep[Section
 2.4]{MacKinnon.Davidson2003}, which says among other things that
 $Q_{\tilde{X}_3} \tilde{X_1} = Q_{[X_2, X_3]}X_1$.
\end{proof}

\begin{lem}\label{lem:sum_pd_psd}
 For any $A \in \SPD{n}$, $B\in \SPSD{n}$, and invertible $C \in \R{n \times n}$,
 \begin{enumerate}
  \item $\tr(C^\tsp [A + B]^{-1} C) \leq \tr(C^\tsp A^{-1}C)$,
  \item $\Vert C^\tsp (A + B)^{-1} C\Vert \leq \Vert C^\tsp A^{-1} C\Vert$,
  \item $\vert C^\tsp (A + B) C\vert \geq \vert C^\tsp A C\vert$, and
  \item $\vert C^\tsp (A + B)^{-1} C\vert \leq \vert C^\tsp A^{-1} C\vert$.
 \end{enumerate}
\end{lem}
\begin{proof}
 All claims can be reduced to the case where $C = I_n$ by either writing $C^\tsp
 (A + B)^{-1}C = (C^{-1}AC^{-\tsp} + C^{-1}BC^{-\tsp}])^{-1}$ and replacing $A$
 and $B$ by $C^{-1}AC^{-\tsp}$ and $C^{-1}BC^{-\tsp}$, respectively, or by
 writing $C^\tsp (A + B)C = C^\tsp A C + C^\tsp  B C$ and replacing $A$ and $B$
 by $C^\tsp A C$ and $C^\tsp B C$, respectively. Assume thus that $C = I_n$.
 Since $A + B$ is SPD, the eigenvalues of $(A + B)^{-1}$ are the reciprocals of
 those of $A + B$. But, letting $\gamma_i(A)$ denote the $i$th eigenvalue in,
 say, decreasing order, Weyl's inequalities \citep[][Corollary
 III.2.2]{Bhatia2012} say $\gamma_i(A + B) \geq \gamma_i(A) + \gamma_n(B) =
 \gamma_i(A) + \emin(B) \geq \lambda_i(A)$, and hence $\tr([A  + B]^{-1}) =
 \sum_{i = 1}^n 1 / \gamma_i(A + B) \leq \sum_{i = 1}^n 1 / \gamma_i(A) =
 \tr(A^{-1})$, which proves the first claim. The remaining claims follow
 similarly since the spectral norm is the maximum eigenvalue for SPSD matrices
 and the determinant is the product of eigenvalues.
\end{proof}

\begin{lem} \label{lem:small_coef}
 For any $X \in \R{n \times p}$, $y \in\R{n}$, and $c > 0$,
 \[
  \Vert (I_p c + X^\tsp X)^{-1}X^\tsp y\Vert \leq \Vert (X^\tsp X)^g
    X^\tsp y\Vert,
 \]
\end{lem}
where superscript $g$ denotes an arbitrary generalized inverse.
\begin{proof}
 Consider the optimization problem of minimizing $g_c:\R{p} \to [0, \infty)$
defined by
 \[
  g_c(b):=\Vert y - Xb\Vert^2 + c \Vert b\Vert^2.
 \]
 If $c = 0$, then any $b$ such that $X^\tsp X b = X^\tsp y$ is a solution. Thus,
 for any generalized inverse, $b_1 = (X^\tsp X)^g X^\tsp y$ solves the problem
 \citep[Theorem 9.1.2]{Harville1997}. On the other hand, if $c > 0$ then since
 $Ic + X^\tsp X$ has full rank, the unique solution is $b_2 = (cI + X^\tsp
 X)^{-1}X^\tsp y$. Now a contradiction arises if for some $c > 0$, $\Vert b_1
 \Vert < \Vert b_2\Vert$, which finishes the proof.
\end{proof}

\begin{lem}\label{lem:g_inv_prod}
  For $A \in \R{n \times n}$ and $B \in \SPD{n}$, we have that $B^{-1}A^gB^{-1}$
  is a generalized inverse of $BAB$, where superscript $g$ indicates a generalized
  inverse.
\end{lem}
\begin{proof}
 We check the definition, namely that $BABB^{-1}A^gB^{-1}BAB = BAB$. Indeed,
 using that $AA^gA = A$, $ BAB B^{-1}A^gB^{-1} BAB = BAA^gAB = BAB$.
\end{proof}

\section{Main results} \label{app:main}

\begin{proof}[Proof Lemma \ref{lem:joint_dens}]
 Let us suppress conditioning on the parameters for simplicity. We have
 \begin{align*}
  f(Y, X) & = f(X)f(Y\mid X)\\
          & = f(X)f(y_1, \dots, y_n \mid X)\\
          & = f(X)f(y_n\mid y_1, \dots, y_{n - 1}, X)f(y_1, \dots,
              y_{n - 1}\mid X)\\
          & =\ \vdots\\
          & = f(X)\prod_{t = 1}^nf(y_t\mid y_1, \dots, y_{t - 1}, X).
 \end{align*}
 Consider an arbitrary term in the product. We have $Y_t = \Alpha^\tsp Z_t +
 \Beta^\tsp X_t + U_t$. Since $Z_t$ is a function of $Y_1, \dots, Y_{t - 1}$,
 both $Z_t$ and $X_t$ are fixed when conditioning on $X$ and $Y_1, \dots, Y_{t -
 1}$. Thus, the distribution of $Y_t \mid X, Y_1, \dots, Y_{t - 1}$ is
 determined by that of $U_t \mid X, Y_1, \dots, Y_{t - 1}$. But $Y_1, \dots,
 Y_{t - 1}$ are functions of $X_1, \dots, X_{t - 1}$ and $U_1, \dots, U_{t -
 1}$, and $\{U_t\}$ is an i.i.d. sequence independent of $\{X_t\}$, and hence
 $X$. Thus, $U_t \mid X, Y_1, \dots, Y_{t - 1} \sim \rN(0, \Sigma)$, and,
 consequently, $Y_t \mid X, Y_1, \dots, Y_{t - 1} \sim \rN(\Alpha^\tsp Z_t +
 \Beta^\tsp X_t, \Sigma)$. Now the result follows by straightforward algebra and
 the fact that the distribution of $\{X_t\}$ does not depend on the model
 parameters.
\end{proof}

\begin{proof}[Proof Proposition~\ref{prop:prop_post}]
 Assuming the posterior is proper, the given expression for the density, up to
 scaling, follows from routine calculations. We prove the posterior is indeed
 proper under either of the two sets of conditions. Since
 \[
  f(Y, X \mid \Alpha, \Beta, \Sigma) = f(Y\mid \Alpha, \Beta, \Sigma, X)f(X\mid
  \Alpha, \Beta, \Sigma) = f(Y\mid \Alpha, \Beta, \Sigma, X)f(X),
 \]
 only the conditional density $f(Y\mid \Alpha, \Beta, \Sigma, X)$ matters when
 deriving the posterior. Under either of the two sets of conditions, $X$ has
 full column rank so $X^\tsp X$ is invertible and we may define $H_X = (X^\tsp
 X)^{-1}X^\tsp$, $P_X = X H_X$, and $Q_X = I_n - P_X$. Let also $E = Y -
 Z\Alpha$ and use $Q_X + P_X = I_n$ to write
 \begin{align} \label{eq:joint_dens}
   & f(Y \mid \Alpha, \Beta, \Sigma, X) \notag \\
   & \propto \vert \Sigma\vert^{-\frac{n}{2}} \etr\left(-\frac{1}{2}
      [X\Beta - E]^\tsp[X\Beta - E]\Sigma^{-1}\right)\notag\\
   & = \vert \Sigma\vert^{-\frac{n}{2}}\etr\left(-\frac{1}{2}[X\Beta - E]^\tsp
      (Q_X + P_X)[X\Beta - E]\Sigma^{-1}\right)\notag\\
   & = \vert \Sigma\vert^{-\frac{n}{2}}\etr\left(-\frac{1}{2} E^\tsp Q_X
      E\Sigma^{-1} \right) \etr\left(-\frac{1}{2}[\Beta - H_X E]^\tsp X^\tsp
      X [\Beta - H_X E]\Sigma^{-1}\right).
 \end{align}
 The right-most term is a kernel of a matrix normal density for $\Beta$ with
 mean $H_X E$ and scale matrices $(X^\tsp X)^{-1}$ and $\Sigma$. Thus,
 integrating with respect to $\Beta$ gives,
 \begin{align*}
  \int f(Y\mid \Alpha, \Beta, \Sigma, X) \, \dd \Beta & \propto \vert
  \Sigma\vert^{-\frac{n}{2}}\etr\left(-\frac{1}{2} E^\tsp Q_X E\Sigma^{-1}
  \right) (2\pi)^{rp/2}\vert X^\tsp X\vert^{-r}\vert \Sigma\vert^{p}\\
  & \propto \vert \Sigma\vert^{-\frac{n - p}{2}}\etr\left(-\frac{1}{2} E^\tsp
  Q_X E\Sigma^{-1} \right).
 \end{align*}
 Thus, to show that $f(Y\mid \Alpha, \Beta, \Sigma, X)f(\alpha)f(\Sigma)$ can be
 normalized to a proper posterior, we need only show that
 \begin{equation} \label{eq:finite_post}
  \iint \vert \Sigma\vert^{-\frac{n - p}{2}}\etr\left(-\frac{1}{2}
  E^\tsp Q_X E\Sigma^{-1} \right)f(\alpha)f(\Sigma)\,\dd\alpha\,\dd\Sigma
  < \infty.
 \end{equation}
 Let us consider the two sets of conditions separately, starting with the first.
 Since \[\tr(E^\tsp Q_X E \Sigma^{-1}) = \tr(\Sigma^{-1/2}E^\tsp Q_X
 E\Sigma^{-1/2}) \geq 0,\] we can upper bound the integrand in
 \eqref{eq:finite_post} by
 \[
  \vert \Sigma\vert^{-\frac{n - p}{2}} f(\alpha)f(\Sigma) =
    \vert \Sigma\vert^{-\frac{n + a - p}{2}}
    \etr\left(-\frac{1}{2}\Sigma^{-1}D\right)f(\alpha),
 \]
 which since we are assuming that $n - p + a - r - 1 > r - 1$, i.e. that $n + a
 > 2r + p$ and that $D$ is SPD, is the product of a proper inverse Wishart and a
 proper density for $\alpha$. This finishes the proof for the first set of
 conditions.

 For the second set of conditions, notice that for \eqref{eq:finite_post} it
 suffices, since $D$ is SPSD, and hence $f(\Sigma)$ and $f(\alpha)$ both
 bounded, to show that
 \[
  \iint \vert \Sigma\vert^{-\frac{n + a - p}{2}}\etr\left(-\frac{1}{2} E^\tsp
  Q_X E\Sigma^{-1} \right)\,\dd\alpha\,\dd\Sigma < \infty.
 \]
 Let $\tilde{Y} = Q_X Y$ and $\tilde{Z} = Q_X Z$ so that $Q_X E = \tilde{Y} -
 \tilde{Z}\Alpha$. Using the same decomposition as before we have for the last
 integrand
 \begin{align*}
   & \vert \Sigma\vert^{-\frac{n + a - p}{2}}\etr\left(-\frac{1}{2}
    E^\tsp Q_X E\Sigma^{-1} \right)\\
   & =  \vert \Sigma\vert^{-\frac{n + a - p - qr}{2}}\etr\left(-\frac{1}{2}
    \tilde{Y}^\tsp Q_{\tilde{Z}}\tilde{Y}\Sigma^{-1} \right)
    \vert \Sigma\vert^{-\frac{qr}{2}} \\
   & \quad \times\etr\left(-\frac{1}{2} [\Alpha - H_{\tilde{Z}}\tilde{Y}]^\tsp
    \tilde{Z}^\tsp \tilde{Z} [\Alpha - H_{\tilde{Z}}\tilde{Y}]\Sigma^{-1}\right).
 \end{align*}
 Under the second set of assumptions, the last line is proportional to the
 product of an inverse Wishart density for $\Sigma$ with scale matrix
 $\tilde{Y}^\tsp Q_{\tilde{Z}}\tilde{Y}$ and $n + a - p - qr - r - 1$ degrees of
 freedom and a matrix normal density for $\Alpha$ with mean
 $H_{\tilde{Z}}\tilde{Y}$ and scale matrices $(\tilde{Z}^\tsp \tilde{Z})^{-1}$
 and $\Sigma$, and hence integrable. The assumption that $[Y, X, Z]$ has full
 column ensures that, by Lemma \ref{lem:partition_rank}, $\tilde{Y}^\tsp
 Q_{\tilde{Z}}\tilde{Y}$ and $\tilde{Z}^\tsp \tilde{Z}$ are positive definite
 matrices.
\end{proof}

\begin{proof}[Proof of Lemma \ref{lem:collapse_dist}]
 The full conditional distribution of $\Beta$ is immediate from dropping terms
 not depending on $\Beta$ in \eqref{eq:joint_dens}. Consider next the integrand
 in \eqref{eq:finite_post}. The first term in the exponential is
 \begin{align*}
  \tr([Y - Z\Alpha]^\tsp Q_X [Y - Z\Alpha]\Sigma^{-1})
  & = \Vert Q_X(Y - Z\Alpha) \Sigma^{-1/2}\Vert_F^2\\
  & = \Vert(\Sigma^{-1/2}\otimes I_n)(\vecop(Q_X Y) -
    \vecop(Q_X Z\Alpha))\Vert^2\\
  & = \Vert(\Sigma^{-1/2}\otimes I_n)(\vecop(Q_X Y) - [I_r \otimes Q_XZ]\alpha)
    \Vert^2.
 \end{align*}
  Thus, the log of the integrand is quadratic as a function of $\alpha$, with
  Hessian $-B = -\Sigma^{-1}\otimes Z^\tsp Q_X Z - C$ and gradient
  $-(\Sigma^{-1} \otimes Z^\tsp Q_X) \vecop(Q_X Y) -Cm$, which implies the
  desired distribution for $\alpha \mid \Sigma, Y$. Finally, the distribution of
  $\Sigma\mid \alpha, Y$ is immediate from dropping terms in the integrand in
  \eqref{eq:finite_post} not depending on $\Sigma$.
\end{proof}

\begin{proof}[Proof Lemma \ref{lem:tv_dist}]
 Assume $\alpha^h, \Beta^h, \Sigma^h$, $h = 1, 2, \dots$ are generated by the
 collapsed Gibbs sampler in Algorithm \ref{alg:collapsed} started at some point
 $\theta^0 \in \Theta$. The equality follows from showing that $\xi^h =
 (\alpha^h, \Sigma^h)$ and $\theta^h$ are co-de-initializing Markov chains
 \citep[Corollary 1]{Roberts.Rosenthal2001}. That they are both Markov chains is
 clear from the construction of the updates in Algorithm \ref{alg:collapsed}.
 That $\theta^h$ is de-initializing for $\xi^h$, i.e. that the distribution of
 $\xi^h \mid \theta^h, \xi^0$ does not depend on $\xi^0$, is immediate from that
 $\xi^h$ is a function (coordinate projection) of $\theta^h$. The other
 direction, that $\xi^h$ is de-initializing for $\theta^h$, is by construction
 of the algorithm: since $\xi^h$ is a coordinate projection of $\theta^h$, the
 distribution of $\theta^h \mid \xi^h, \theta^0$ is determined by that of
 $\Beta^h \mid \xi^h, \theta^0$, and the distribution from which this value is
 drawn (line 4, Algorithm \ref{alg:collapsed}) does not depend on $\theta^0$.
 Similarly, notice that the distribution of $\xi^h \mid \xi^{h - 1}$ is the same
 as $\xi^h \mid \alpha^{h - 1}$ by construction of the algorithm. Thus,
 $\alpha^{h - 1}$ is de-initializing for $\xi^h$ and the inequality follows
 \citep[Theorem 1]{Roberts.Rosenthal2001}.
\end{proof}

\subsection{Inadequacy of Drift Function in Theorem~\ref{thm:geo_erg}}
\label{app:inadequacy}

\begin{prop} \label{prop:L}
 For the $L = L_n(Y, X)$ defined in Lemma \ref{lem:drift} it holds for some $c_1
 > 0$ and $c_2 > 0$, depending on the hyperparameters but not the data, that
 \[\Vert \hat{\alpha}\Vert_F^2 \leq L_n(Y, X) \leq c_1 \Vert
 \hat{\alpha}\Vert_F^2+ c_2,\] and hence $L_n(Y, X) = O_\pr(1)$ if and only if
 $\hat{\alpha} = O_\pr(1)$.
\end{prop}

\begin{proof}
 Since $C^{-1}$ is SPD, the term $\tr(C^{-1})$ is positive, and so dropping it
 and the term $\Vert C^{-1}\Vert \Vert Cm\Vert$ in the expression for $L_n(Y,
 X)$ gives
 \[
  L_n(Y, X) > \Vert C ^{-1/2}\Vert^2 \Vert C^{1/2}\hat{\alpha}\Vert^2 =
  \emin(C)^{-1}\Vert C^{1/2}\hat{\alpha}\Vert^2.
\]
 On the other hand, using that $(r_1 + r_2)^2 \leq 2r_1^2 + 2r_2^2$ for any real
 numbers $r_1$ and $r_2$,
 \[
  L_n(Y, X) \leq 2 \Vert C^{-1}\Vert^2 \Vert Cm\Vert^2 + 2 \Vert C^{-1}\Vert
  \Vert C^{1/2}\hat{\alpha}\Vert^2 + \tr(C^{-1}).
 \]
 Now notice that $C - \emin(C) I_{qr^2}$ and $\emax(C)I_{qr^2} - C$ are both
SPSD, and therefore
 \[
  \emin(C) \Vert \hat{\alpha}\Vert^2 \leq \hat{\alpha}^\tsp C \hat{\alpha} \leq
  \emax(C) \Vert \hat{\alpha}\Vert^2.
 \]
 Thus, since $0 < \emin(C) < \emax(C) < \infty$ and $\hat{\alpha} =
 \vecop(\hat{\Alpha})$, we are done.
\end{proof}

\begin{prop}\label{prop:eps_vanish}
 If, almost surely as $n\to \infty$,
 \[
  n \Vert Q_X Z\Vert^2/\Vert Q_{[Z, X]} Y \Vert^2 \to \infty,
 \]
 then the $\veps = \veps_n$ in Theorem \ref{thm:geo_erg} tends to zero almost
surely as $n\to \infty$. In particular, $\veps_n \to 0$ almost surely if $n^{-1}
Y^\tsp Q_{[Z, X]}Y$ and $n^{-1} Z^\tsp Q_X Z$ have positive definite limits
almost surely.
\end{prop}

\begin{proof}
 Recall from Lemma \ref{lem:minor} the definition of $\veps = \veps_n$, $c = n +
a - p - r - 1$, and $c_1 = (\Vert Q_X Y\Vert + \Vert Q_XZ\Vert \sqrt{T})^2$. It
suffices to show that $\zeta_n := \veps_n^{2 n/ c} \to 0$ almost surely since
$2n/c \to 2$. We have
 \[
  \zeta_n = \left[ \frac{\vert D + Y^\tsp Q_{[Z, X]}Y\vert}{\vert D + I_r
  c_1\vert }\right]^n.
 \]
 By Lemma \ref{lem:sum_pd_psd}, $\vert D + Y^\tsp Q_{[Z, X]}Y\vert \leq \vert D
 + I_r \Vert Q_{[Z, X]} Y\Vert^2 \vert$. Thus,
 \begin{align*}
  \zeta_n \leq \left[ \frac{\vert D + I_r \Vert Q_{[Z, X]}Y\Vert^2\vert}{\vert D
  + I_r c_1\vert }\right]^n.
 \end{align*}
 Now using that $c_1 \geq \Vert Q_X Y\Vert^2 + \Vert Q_X Z\Vert^2 T$ and $\Vert
 Q_X Y\Vert \geq \Vert Q_{[Z, X]} Y\Vert$ in another application of Lemma
 \ref{lem:sum_pd_psd},
 \[
  \zeta_n \leq \left[ \frac{\vert D + I_r \Vert Q_{X}Y\Vert^2\vert}{\vert D +
  I_r \Vert Q_XY\Vert^2 + I_r \Vert Q_X Z\Vert^2 T \vert }\right]^n,
 \]
 which can be written as a product of $r$ terms, the $j$th of which is
 \[
  \left(1 + \frac{d_j + T \Vert Q_X Z\Vert^2}{d_j + \Vert Q_{[Z, X]} Y\Vert^2
  }\right)^{-n},
 \]
 where $d_j$ is the $j$th eigenvalue of $D$. Since $r$ is fixed, the product of
 $r$ terms tends to is zero if and only if one of the terms does, which happens
 unless $T \Vert Q_X Z\Vert^2 / \Vert Q_{[Z, X]}Y\Vert^2 = O(1/n)$ since $d_j$
 is fixed.
\end{proof}

\end{document}